\documentclass[11pt]{article}
\usepackage[utf8]{inputenc}
\usepackage[margin=1.1in,a4paper]{geometry}

\usepackage{mathtools}
\usepackage{amsmath, amssymb, graphicx, url, amsthm, subcaption,dsfont}
\usepackage[ruled, lined]{algorithm2e}

\usepackage[mathcal]{eucal}
\usepackage{tikz}
\usepackage{verbatim}
 \usepackage{tcolorbox}
 
\usepackage{amsfonts}       
\usepackage{enumitem}
\usepackage{microtype}      
\usepackage{xcolor}
\usepackage[colorlinks=true,citecolor=blue]{hyperref}   
\usepackage{diagbox}   
\usepackage{tcolorbox}
\usepackage{nicefrac}

\newcommand{\RR}{\mathbb{R}}
\newcommand{\NN}{\mathbb{N}}

\newcommand{\Cc}{\mathcal{C}}

\newcommand{\Pp}{\mathcal{P}}

\newcommand{\Xx}{\mathcal{X}}

\newcommand{\E}{\mathbf{E}}

\renewcommand{\d}{\mathrm{d}}

\newcommand{\id}{\mathrm{id}}
\newcommand{\Id}{\mathrm{Id}}

\newcommand{\note}[1]{{\textbf{\color{red}#1}}}

\newtheorem{theorem}{Theorem}[section]
\newtheorem{proposition}[theorem]{Proposition}
\newtheorem{lemma}[theorem]{Lemma}
\newtheorem{corollary}[theorem]{Corollary}

\newtheorem{definition}[theorem]{Definition}

\DeclareMathOperator{\Var}{Var}

\DeclareMathOperator{\cost}{cost}
\DeclareMathOperator{\bmu}{\boldsymbol{\mu}}
\DeclareMathOperator{\bnu}{\boldsymbol{\nu}}
\DeclareMathOperator{\bphi}{\boldsymbol{\phi}}
\DeclareMathOperator{\bh}{\boldsymbol{h}}
\DeclareMathOperator{\bg}{\boldsymbol{g}}
\DeclareMathOperator{\bW}{\boldsymbol{W_2}}
\graphicspath{ {./images/} } 
\providecommand{\keywords}[1]
{
  \small	
  \textbf{\textit{Keywords---}} #1
}
\providecommand{\MSC}[1]
{
  \small	
  \textbf{\textit{MSC Classification---}} #1
}
\newcommand{\rev}[1]{{ #1}}

\title{Displacement smoothness of entropic optimal transport}

\author{
Guillaume Carlier\thanks{Ceremade, Université Paris Dauphine, PSL, 75775, Paris, and Inria-Paris, Mokaplan}
\and
L\'ena\"ic Chizat\thanks{Institute of Mathematics, École polytechnique fédérale de Lausanne (EPFL), Station Z, CH-1015 Lausanne}
\and
Maxime Laborde\thanks{Université Paris Cité and Sorbonne Université, CNRS, Laboratoire Jacques-Louis Lions (LJLL), F-75006 Paris, France}
}

\begin{document}
\maketitle

\begin{abstract}
The function that maps a family of probability measures to the solution of the dual entropic optimal transport problem is known as the Schr\"odinger map.
We prove that when the cost function is $\Cc^{k+1}$ with $k\in \NN^*$ then this map is Lipschitz continuous from the $L^2$-Wasserstein space to the space of $\Cc^k$ functions. 
Our result holds on compact domains and covers the multi-marginal case. \rev{We also include regularity results under negative Sobolev metrics weaker than Wasserstein under stronger smoothness assumptions on the cost.}
As applications, we prove displacement smoothness of the entropic optimal transport cost and the well-posedness of certain Wasserstein gradient flows involving this functional, including the Sinkhorn divergence and a multi-species system.
\end{abstract}
\keywords{Entropic Optimal Transport, Schr\"odinger map, Wasserstein gradient flows.}\\
\noindent
\MSC{ 49Q22, 49K40, 35A15.}

\section{Introduction}\label{sec:introduction}

The main goal of this paper is to study the regularity of the multi-marginal Entropic Optimal Transport (EOT) problem under ``displacement'' of the marginals, and to apply these results to prove the well-posedness of certain evolution equations and optimization methods involving EOT. For clarity of presentation, let us first present the context and our results in the classical two marginals case. Let $\Xx_1,\Xx_2 \subset \RR^d$ be two compact convex sets, $\bmu = (\mu_1,\mu_2) \in \Pp(\Xx_1)\times \Pp(\Xx_2)$ two probability measures and $c \in \Cc^k(\Xx_1\times \Xx_2)$ a  $k$-times continuously differentiable \emph{cost} function. We consider the \emph{entropic optimal transport} problem defined as
\begin{align}\label{eq:EOT}
   E(\mu_1,\mu_2) \coloneqq \min_{\gamma \in \Pi(\mu_1,\mu_2)} \int c(x_1,x_2)\d\gamma(x_1,x_2) + H(\gamma|\mu_1\otimes \mu_2)
\end{align}
where $\Pi(\mu_1,\mu_2)$ is the set of \emph{transport plans} between $\mu_1$ and $\mu_2$, that is probability measures on $\Xx_1\times \Xx_2$ with marginals $\mu_1$ and $\mu_2$, and $H$ is the \emph{relative entropy} defined as $H(\mu|\nu) = \int \log(\d \mu / \d\nu)\d\mu$ if $\mu$ is absolutely continuous w.r.t.~$\nu$ and $+\infty$ otherwise.

This problem can be seen as a regularization of the optimal transport problem~\cite{villani2009optimal, santambrogio2015optimal} that benefits from improved computational~\cite{kosowsky1994invisible, cuturi2013sinkhorn} and statistical properties~\cite{genevay2019sample,mena2019statistical,del2022improved}, at the expense of an approximation error that can be quantified~\cite{pal2019difference,weed2018explicit,chizat2020faster,conforti2021formula, CPT22, Nutz22}.  It is also tightly related to the Schr\"odinger bridge problem~\cite{schrodinger1932theorie,leonard2012schrodinger}, which is a modification of Eq.~\eqref{eq:EOT} obtained by replacing $\mu_1\otimes \mu_2$ by the product Lebesgue measure in the relative entropy term.

\subsection{Schr\"odinger Map}
Eq.~\eqref{eq:EOT} defines a convex optimization problem which admits a dual concave maximization formulation
\begin{align*}
   E(\mu_1,\mu_2) = \max_{\substack{\phi_1\in \Cc^0(\Xx_1)\\ \phi_2\in \Cc^0(\Xx_2)}} \int_{\Xx_1}\!\! \phi_1\d\mu_1 +\int_{\Xx_2}\!\! \phi_2\d\mu_2 + 1 - \int_{\Xx_1\times \Xx_2}\!\!\!\! e^{\phi_1(x_1)+\phi_2(x_2)-c(x_1,x_2)}\d\mu_1(x_1)\d\mu_2(x_2).
\end{align*}
This dual problem admits solutions which satisfy, for $\mu_1\otimes \mu_2$ almost every $(x_1,x_2)$, the following first order optimality conditions, known as the \emph{Schr\"odinger system}:
\begin{align*}
    \left\{
    \begin{aligned}
    \phi_1(x_1) &= -\log \int_{\Xx_2} e^{\phi_2(x_2)-c(x_1,x_2)}\d\mu_2(x_2)\\
    \phi_2(x_2) &= -\log \int_{\Xx_1} e^{\phi_1(x_1)-c(x_1,x_2)}\d\mu_1(x_1)
    \end{aligned}
    \right. .
\end{align*}
In this paper, our main object of interest is the particular solution $(\phi_1,\phi_2)$ which satisfies the Schr\"odinger system \emph{for every} $(x_1,x_2)\in \Xx_1\times \Xx_2$. It is not difficult to see that this solution inherits the $\Cc^k$ regularity of $c$ and is unique in the quotient space $\tilde \Cc^k \coloneqq \Cc^k(\Xx_1)\times \Cc^k(\Xx_2)/\sim$ where the equivalence relation 
\[
(\phi_1,\phi_2)\sim(\psi_1,\psi_2)\quad \Leftrightarrow\quad \exists \kappa \in \RR \text{ such that } \phi_1 =\psi_1+\kappa \text{ and } \phi_2 = \psi_2-\kappa
\]
captures the trivial invariance of the dual problem.
This particular choice of solution $(\phi_1,\phi_2)$ is arguably the most natural to consider thanks to its stability.  It is also useful in many contexts because it represents the differential of the  functional $E$~\cite{feydy2019interpolating}. We refer to this special solution $(\phi_1,\phi_2)$ as the \emph{Schr\"odinger potentials} and we define the \emph{Schr\"odinger map} $S:\Pp(\Xx_1)\times \Pp(\Xx_2) \to \tilde \Cc^k$ as
\begin{align}\label{def-S-intro}
    S:\bmu=(\mu_1,\mu_2)\mapsto \bphi=(\phi_1,\phi_2).
\end{align}

\subsection{Main result in the two-marginals case}
Our main contribution is a proof that the Schr\"odinger map $S$ is Lipschitz continuous with respect to the following distances:
\begin{itemize}
    \item We endow $\Pp(\Xx_i)$ with the Wasserstein metric defined for two probability measures $\mu, \nu \in \Pp(\Xx_i)$ by 
    \[
W_2(\mu,\nu) \coloneqq \left(\min_{\gamma \in \Pi(\mu,\nu)} \int_{\Xx_i \times \Xx_i} \Vert y-x\Vert^2\d\gamma(x,y)\right)^{\frac12}.
\]
 and then we endow  $\Pp(\Xx_1)\times \Pp(\Xx_2)$ with the product Wasserstein  $\bW$ metric given for $\bmu =(\mu_1,\mu_2),\, \boldsymbol{\nu} =(\nu_1, \nu_2) \in \Pp(\Xx_1)\times\Pp(\Xx_2) $
 \[
 \bW(\bmu,\boldsymbol{\nu}) \coloneqq \left(W_2(\mu_1,\nu_1)^2 +W_2(\mu_2,\nu_2)^2 \right)^{\frac12}.
 \]

\item We endow $\tilde \Cc^k$ with the product, quotient supremum $\Cc^k$ norm defined from the usual $\Cc^k$ norm $\Vert \cdot \Vert_{\Cc^k}$ as
\[
\Vert (\phi_1,\phi_2)\Vert_{\tilde \Cc^k} \coloneqq \inf_{\kappa \in \RR} \Vert \phi_1-\kappa\Vert_{\Cc^k} + \Vert \phi_2+\kappa\Vert_{\Cc^k} .
\]
\end{itemize}

Our Lipschitz stability result for the Schr\"odinger map -- which is a particular case of the more general results Thm.~\ref{thm:main} and  Cor.~\ref{cor:lipschitz-multi} that cover the multi-marginal case and finer regularity results -- reads as follows (where by convention $\NN^*:=\NN\setminus \{0\}$).
\begin{theorem}\label{thm:main-2marginals}

If $c\in \Cc^{k+1}(\Xx_1\times \Xx_2)$ for $k\in \NN^*$, then there exists $C>0$ that only depends on $\Vert c\Vert_{\Cc^{k+1}}$ such that for all $\bmu,\bmu' \in \Pp(\Xx_1)\times \Pp(\Xx_2)$,  
$$
\Vert S(\bmu) - S(\bmu') \Vert_{\tilde \Cc^k} \leq C  \bW(\bmu,\bmu').
$$
\end{theorem}

This translates into a useful regularity result for the functional $E$ (see Thm.~\ref{thm:smoothness}).

\begin{corollary}
If $c\in \Cc^{2}(\Xx_1\times \Xx_2)$, then given $(\mu^t_1)_{t\in [0,1]}$ and $(\mu^t_2)_{t\in [0,1]}$ two Wasserstein geodesics, the map $h:t \mapsto E(\mu_1^t,\mu_2^t)$ is differentiable and its derivative satisfies
$$
\vert h'(t)-h'(s)\vert \leq C \vert t-s\vert \bW(\bmu^0, \bmu^1) 
$$
where $\bmu^t= (\mu_1^t , \mu_2^t)$ and $C>0$ only depends on $\Vert c\Vert_{\Cc^2}$. In particular, $E$ and $-E$ are displacement semi-convex.
\end{corollary}

As an application of these results, we will prove the well-posedness of Wasserstein gradient flows for several energies involving the functional $E$ in Section~\ref{sec:applications}, and also establish exponential convergence to  equilibrium in some cases. Since the Wasserstein gradient of $E$ is  $\nabla S_1, \nabla S_2$ where  $S=(S_1,S_2)$ is defined in  \eqref{def-S-intro}, we have for example the following result (see Prop.~\ref{prop:multi-species}), where $H(\mu)\coloneqq H(\mu|\mathrm{Leb})$ is the (convex) differential entropy.
\begin{corollary}
Let $c \in \mathcal{C}^2(\Xx_1 \times \Xx_2)$ and $\bmu^{0}=(\mu_1^0, \mu_2^0) \in \Pp(\Xx_1)\times \Pp(\Xx_2)$. Then the functional $F$ defined by
\[
F(\bmu) := E(\bmu) + H(\mu_1) + H(\mu_2),
\]
admits a unique Wasserstein gradient flow starting from $\bmu^0$, i.e.~there exists a unique absolutely continuous curve $\bmu^t =(\mu_{1}^t, \mu_{2}^t)\in \Pp(\Xx_1)\times \Pp(\Xx_2)$ for $\bW$, satisfying
\[
\left\{ \begin{aligned}
\partial_t \mu_{1} &= \nabla \cdot (\mu_1 \nabla S_1(\bmu)) +  \Delta \mu_1\\
\partial_t \mu_{2} &= \nabla \cdot (\mu_2 \nabla S_2(\bmu)) + \Delta \mu_2 \\
\bmu_{\vert_{t=0}}&=\bmu^0
\end{aligned}
\right.
\]
with no-flux boundary conditions.
If in addition $H(\mu_1^0), H(\mu_2^0) <+\infty$, then $\bmu^t$ converges at an exponential rate to the unique equilibrium  $\bmu^*$ (see Eq.~\eqref{defdemustar}), in the sense that there exists $\kappa >0$ independent of $\bmu^0$ such that
\[
F(\bmu^t) - F(\bmu^*) \leq e^{-\kappa t}(F(\bmu^0) - F(\bmu^*)).
\]
\end{corollary}

Let us mention that the system of PDEs in the previous corollary may naturally appear as an evolution model for cities: $\mu_1$ represents the distribution of agents, $\mu_2$ the distribution of firms, $S_2$ the wage paid by firms to agents and the fact that $S_1$ and $S_2$ are given by \eqref{def-S-intro} captures an equilibrium condition on the labour market. For more details about such models, we refer to \cite{Laborde2020Nonlinear} (for a gradient flow approach without entropic regularization) and to \cite{Barilla2021mean} (in the different context of mean-field games). For extensions to more than two species and more general functionals (typically $F+G$ where $G$ is displacement convex), see Section~\ref{sec:applications}.

More applications of Thm.~\ref{thm:main-2marginals} are developed in companion papers that study optimization dynamics for trajectory inference~\cite{chizat2022trajectory} and regularized Wasserstein barycenters~\cite{chizat2023doubly}, also involving the functional $E$.

\subsection{Discussion of prior work}
 
Several works have studied the stability of the unregularized optimal transport problem~\cite{delalande2021quantitative,berman2021convergence,gigli2011holder}. In particular, it is known that with the square-distance cost, the Kantorovich potential from $\mu_1$ to $\mu_2$ (i.e.~the counterpart of the Schr\"odinger potential $\phi_1$ in unregularized optimal transport) is a $\frac12$-H\"older function of $\mu_2$ from $W_2$ to $\dot H^1(\mu_1)$ under suitable assumptions on the fixed reference measure $\mu_1$, and that this is the strongest regularity that one can hope for in general~\cite{gigli2011holder}.

In~\cite{carlier2020differential}, it is proved by inverse function arguments that $S$ is Lipschitz continuous and smooth as a map from $L^\infty_{++}\to L^\infty$, given some fixed reference measures on the ambient space. Our results will use a similar strategy but \rev{in contrast to~\cite{carlier2020differential} and other follow-up works such as~\cite[Thm.~3]{goldfeld2022limit}, we do not consider stability under additive perturbations of the marginals, but under displacement perturbations by changing the parametrization of the problem.
This leads to  the stronger conclusion that $E$ is smooth in Wasserstein distance which, 
}
as we shall see, is  particularly useful in the context of gradient flows of energies involving $E$. 
Let us also mention that much less is known about stability in the non-compact case, see e.g.~\cite{eckstein2021quantitative} that shows Lipschitz continuity of $E$, ~\cite{nutz2022stability} where the continuity of the Schr\"odinger map for the topology of convergence in probability in a certain non-compact setting, see also~\cite{ghosal2021stability} for further stability results for the primal variable. Our results are finer, but rely in an essential way on the compact setting.

\rev{Note that a result equivalent to Thm.~\ref{thm:main-2marginals} for $k=0$ was already proved in~\cite{deligiannidis2021quantitative}. Their elegant approach consists in showing that the Sinkhorn's iteration is stable under $W_1$ perturbations (or, equivalently in the compact setting, $W_2$ perturbations) of the marginals which, combined with the fact that this iteration is a contraction for the so-called Hilbert metric, leads to the conclusion.
The strength of their analysis is that it applies to $k=0$ (i.e.~merely Lipschitz continuous costs); and from the result with $k=0$, it is not difficult to prove the $k\geq 1$ case under regularity assumptions on the cost. However, their proof technique would likely not extend to the multi-marginal case (a well-know limitation of the Hilbert metric approach). Here, we propose an independent proof technique for all $k\geq 1$ (for displacement smoothness, we need the case $k=1$) and our analysis also gives additional information on higher degrees of smoothness of the Schr\"odinger map and of $E$. }

The rest of the paper is organized as follows. In Section~\ref{sec:multi-marginal}, we introduce the multi-marginal setting and state the full version of our regularity results for the Schr\"odinger map. The proofs of those statements can be found in Section~\ref{sec:proofs}. Finally, in Section~\ref{sec:applications} we study the regularity of the functional $E$ and apply our results to the analysis of certain Wasserstein gradient flows involving $E$.

\section{The multi-marginal case}\label{sec:multi-marginal}
\paragraph{Notation and assumptions on the domain}
Let $N\geq 2$ be the number of marginals, let $\Xx_i \subset \RR^d$ be convex and  compact, for $i\in [N]\coloneqq \{1,\dots,N\}$, and let $\Xx\coloneqq\prod_{i=1}^N \Xx_i$. Given $i\in [N]$, we denote $\Xx_{-i}=\prod_{j\neq i} \Xx_j$ and identify $\Xx$ to $\Xx_i\times \Xx_{-i}$, i.e.~we denote $x=(x_1,\dots,x_N)$ as $x=(x_i,x_{-i})$.

For $k\geq 0$, let $\Cc^{k}(\Xx_i)$ be the space of $k$-times continuously differentiable functions over $\Xx_i$ (that is, functions defined on $\Xx_i$ that admit a $\Cc^k$ extension on $\RR^d$) endowed with the usual supremum norm. Using the multi-index notation, this norm is defined as
\(
\Vert f\Vert_{\Cc^k} \coloneqq \inf_{\tilde f}\sup_{\vert \alpha\vert \leq k} \Vert \tilde f^{(\alpha)}\Vert_\infty
\)
where the infimum is over functions $\tilde f$ that are extensions of $f$ defined on $\RR^d$. Endowed with this norm, $\Cc^k(\Xx_i)$ is a Banach space, see~\cite[Chap.~8, II]{queffelec2020analyse} for details.

We denote by $\Pp(\Xx_i)$ the space of Borel probability measures on $\Xx_i$, which we endow with the weak topology, characterized by its convergent sequences as $\mu_n \rightharpoonup \mu \Leftrightarrow \int \phi\d\mu_n\to \int \phi\d\mu$ for all $\phi\in \Cc^0(\Xx_i)$. 
Given $(\mu_1, \dots, \mu_N) \in \prod_{i=1}^N \Pp(\Xx_i)$, we denote by $\bmu$ the $N$-tuple $(\mu_1, \dots, \mu_N)$ and by $\mu$ the product measure $\otimes_{i=1}^N \mu_i \in \Pp(\Xx)$. For $\mu = \otimes_{i=1}^N \mu_i \in \Pp(\Xx)$ we let $ \mu_{-i} \coloneqq \otimes_{j\neq i}\mu_j$
and $\prod_{i=1}^N \Pp(\Xx_i)$ is endowed with the product Wasserstein  $\bW$ metric, given, for $\bmu =(\mu_1, \dots, \mu_N),\, \boldsymbol{\nu} =(\nu_1, \dots, \nu_N) \in \prod_{i=1}^N \Pp(\Xx_i)$, by
 \[
 \bW(\bmu,\boldsymbol{\nu}) \coloneqq \left( \sum_{i=1}^N W_2(\mu_i,\nu_i)^2 \right)^{\frac12}.
 \]
 
In the following, $k\in \NN^*$ is arbitrary and always denotes the regularity of the output space $\tilde \Cc^k$ of the Schr\"odinger map $S$, while the regularity we require for the cost function $c$ varies across statements.

In the proofs, we use $C,C',\dots$ to denote positive constants that may change from line to line and only depend on general characteristics of the problem such as $N$ and other quantities that are specified when needed.

\subsection{Multi-marginal Schr\"odinger System}

The multi-marginal Schr\"odinger system arises as the optimality conditions for the multi-marginal \emph{entropic optimal transport} problem, defined for $\bmu \in \prod_{i=1}^N\Pp(\Xx_i)$ by
\begin{align}
E(\bmu) \coloneqq \min_{\gamma\in \Pi(\bmu)} \int_\Xx c(x)\d\gamma(x) + H(\gamma|\mu).
\end{align}
where $\Pi(\bmu)$ is the set of probability measures on $\Xx$ having marginals $(\mu_1,\dots,\mu_N)$. This convex problem admits a concave dual formulation in terms of the Lagrange multipliers for the marginal constraints

\begin{align}\label{eq:dual-MEOT}
E(\bmu)=\max_{\bphi \in \prod_{i=1}^N \Cc^0(\Xx^i)} \sum_{i=1}^N \int_{\Xx_i} \phi_i(x_i)\d\mu_i(x_i) + 1 - \int_\Xx e^{\sum_{i=1}^N \phi_i(x_i)-c(x)}\d\mu(x).
\end{align}

Notice that these problems are multi-marginal generalizations of those presented in Section~\ref{sec:introduction}. The primal-dual optimality conditions read

\begin{align}\label{eq:opt-primal-dual}
\gamma(\d x) = e^{\sum_{i=1}^N\phi_i(x_i)-c(x)}\mu(\d x).
\end{align}

We refer to~\cite{leonard2012schrodinger, nutz2021introduction} for the basic theory of entropic optimal transport and~\cite{carlier2020differential, marino2020optimal} for the multi-marginal theory. The optimality conditions for Eq.~\eqref{eq:dual-MEOT} coincide with the condition that $\gamma\in \Pi(\bmu)$ in Eq.~\eqref{eq:opt-primal-dual}, and lead to the Schr\"odinger system.

\begin{definition}[Schr\"odinger system/potentials/map]\label{def:schrodinger}
Consider the map $T: \prod_{i=1}^N\Cc(\Xx_i)\times \prod_{i=1}^N\Pp(\Xx_i)\to \prod_{i=1}^N\Cc(\Xx_i)$ defined for $i\in [N]$  and $x_i\in \Xx_i$ as
\begin{align}
T_i(\bphi,\bmu)(x_i) &\coloneqq \log \Big( \int_{\Xx_{-i}} e^{\sum_{j=1}^N\phi_j(x_j)-c(x_i,x_{-i})}\d\mu_{-i}(x_{-i})\Big).
\end{align}
A function $\bphi=(\phi_1, \dots, \phi_N)$ is called a \emph{Schr\"odinger potential} associated to $\bmu$ if it solves the \emph{Schr\"odinger system}
\begin{align}\label{eq:schrodinger-system}
T(\bphi,\bmu) = 0.
\end{align}
The \emph{Schr\"odinger map} is the function $S$ that maps $\bmu$ to its Schr\"odinger potential $\bphi$, i.e.~that satisfies $T(S(\bmu),\bmu)=0$ (Prop.~\ref{prop:preliminaries} states that this map is well-defined in a suitable sense).
\end{definition}

Let us stress that we require Eq.~\eqref{eq:schrodinger-system} to hold in the space of continuous functions, that is for \emph{every} $x\in \Xx$, rather than only $\mu$-a.e.~which is the optimality condition of Eq.~\eqref{eq:dual-MEOT}. 

Clearly, if $\bphi=(\phi_1,\dots,\phi_N)$ solves~\eqref{eq:schrodinger-system} for some fixed $\bmu$, then so does every family of potentials of the form $(\phi_1+\kappa_1,\dots,\phi_N+\kappa_N)$ where the $\kappa \in \RR^{N}$ satisfies  $\sum_{i=1}^N \kappa_i=0$. This defines an equivalence relation $\sim$ and we define the quotient space $$\tilde \Cc^k \coloneqq \Big(\prod_{i=1}^N \Cc^k(\Xx_i)\Big)/\sim.$$ Endowed with the quotient norm (the infimum of the norm over all representatives in the equivalence class), $\tilde \Cc^k$ is a Banach space.

\subsection{Existence and weak continuity of the Schr\"odinger map}

Let us state some preliminary results about the Schr\"odinger map. 

\begin{proposition}\label{prop:preliminaries}
If $c\in \Cc^k(\Xx)$ for $k\in \NN^*$ then for any $\bmu \in \prod_{i=1}^N \Pp(\Xx_i)$, there exists a unique $\bphi=\bphi^{\bmu} \in \tilde \Cc^k$ such that $T(\bphi,\bmu)=0$, i.e.~the Schr\"odinger map $S:\bmu\mapsto \bphi^{\bmu}$ is well-defined. Moreover, 
\begin{itemize}
    \item[(i)] for $i\in [N]$, $\phi^{\bmu}_i$ is $L_i$-Lipschitz continuous, where $L_i=\sup_{x\in \Xx} \Vert \nabla_{x_i} c(x)\Vert_2$,
    \item[(ii)] the Schr\"odinger map $S:\prod_{i=1}^N \Pp(\Xx_i)\to \tilde \Cc^k$ is weakly continuous and is bounded.
    \end{itemize}
\end{proposition}
The continuity claim $(ii)$ is not needed in the sequel -- and is weaker than Thm.~\ref{thm:main} -- but it is instructive to recall this known result that can be obtained by elementary means, before delving into more technical proofs.
\begin{proof}
The existence of a unique solution to Eq.~\eqref{eq:schrodinger-system} in $\prod_{i=1}^N L^\infty(\mu_i)/\sim$, i.e.~in the $\mu$-almost-everywhere sense, is proved in~\cite{carlier2020differential}, see also~\cite{marino2020optimal}. In order to prove the same in $\tilde \Cc^0$, i.e.~in the everywhere sense, let us observe that $T$ can be expressed as $T=\Id+\bar T$ with
\begin{align}\label{eq:Tbar}
\bar T_i(\bphi,\bmu)(x_i) = \log \Big( \int_{\Xx_{-i}} e^{\sum_{j\neq i}\phi_j(x_j)-c(x_i,x_{-i})}\d\mu_{-i}(x_{-i})\Big)
\end{align}
by factorizing $e^{\phi_i(x_i)}$ out in the definition of $T_i$. Thus, given a representer of the $L^\infty$ solution $\bphi^{L^\infty}\in \prod_{i=1}^N L^\infty(\Xx_i)$, one can \emph{define} a solution $\bphi^{\Cc^k} \in \prod_{i=1}^N \Cc^k(\Xx_i)$ in the ``everywhere sense'' by setting
\begin{align}\label{eq:splitting}
\phi_i^{\Cc^k}(x_i) \coloneqq -\log \Big( \int_{\Xx_{-i}} e^{\sum_{j\neq i}\phi^{L^\infty}_j(x_j)-c(x_i,x_{-i})}\d\mu_{-i}(x_{-i})\Big) = - \bar T_i(\bphi^{L^\infty},\bmu)(x_i) .
\end{align}
Observe that $\bphi^{\Cc^k}$ inherits the $\Cc^k$ regularity of $c$.
Moreover, by uniqueness in $L^\infty$, any ``everywhere'' solution must coincide $\mu$-a.e. with an ``almost everywhere'' solution and is thus of the form given by Eq.~\eqref{eq:splitting}. Noticing that $\bar T(\bphi+\kappa,\bmu) =  \bar T(\bphi,\bmu)-\kappa$ for any family of constants $\kappa\in \RR^N$ such that $\sum_{i=1}^N \kappa_i=0$ and quotienting by $\sim$, it follows that there exists a unique solution of the Schr\"odinger system~\eqref{eq:schrodinger-system} in $\tilde \Cc^k$.

The Lipschitz continuity constant of $x_i\mapsto \phi^{\bmu}_i(x_i)$ can be bounded by observing that differentiating $\bphi^{\bmu} = -\bar T(\bphi^{\bmu},\bmu)$ in $x_i$ gives
\begin{align*}
    \nabla \phi_i^{\bmu}(x_i) =\int_{\Xx_{-i}} \nabla_{x_i} c(x_i,x_{-i})\d Q_{-i}(x_{-i}|x_i) 
\end{align*}
where $Q_{-i}(\cdot|x_i)\in \Pp(\Xx_{-i})$ is a probability measure whose expression is given later in Eq.~\eqref{eq:def-Q}.

Finally, to prove weak continuity of $S$, it is enough to prove that if $\bmu^n$ is a sequence weakly converging to $\bmu$, then $\bphi^{\bmu^n}$ converges to $\bphi^{\bmu}$. By the previous point, $\{ \bphi^{\bmu^n}\}_n$ is uniformly Lipschitz continuous and one can choose a uniformly bounded sequence of representatives so by Ascoli-Arzel\`a Theorem we can extract a subsequence $\bphi^{m_n}$ which converges in $\tilde \Cc^0$ to some $\bphi^\infty$. Since the map $T$ is jointly continuous, it follows that $T(\bphi^\infty,\bmu)=\lim_m T(\bphi^{\bmu^{m_n}},\bmu^{m_n})=0$, hence $\bphi^\infty=\bphi^{\bmu}$. But this limit is unique, so the full sequence $(\bphi^{\bmu^n})_n$ converges to $\bphi^{\bmu}$ which proves the weak continuity in $\tilde \Cc^0$. Since one also has $\bphi^{\bmu^n} = -\bar T(\bphi^{\bmu^n},\bmu^n)$ for all $n$ and $\bar T$ is continuous  as a function $\tilde \Cc^0\times \prod_{i=1}^N \Pp(\Xx_i)\to\tilde \Cc^k$, we have in fact that $\bphi^{\bmu^n}$ converges to $\bphi^{\bmu}$ in $\tilde \Cc^k$. Boundedness of $S$ finally follows from the fact that it is weakly continuous on a weakly compact set.
\end{proof}

\subsection{Main result : regularity of the Schr\"odinger map}\label{sec-main-result}
In order to study regularity beyond the zero-th order, we bypass the lack of differentiable structure of $\Pp(\Xx)$ by considering parametrized paths generated by transport plans.

Consider $\bmu^0 = (\mu^0_i)_{i=1}^N$ and $\bmu^1 = (\mu^1_i)_{i=1}^N$ two families of probability measures in $\prod_{i=1}^N\Pp(\Xx_i)$, and a family of transport plans\footnote{To be clear, we \emph{do not} assume that $\gamma_i$ is an \emph{optimal} transport plan.} $\gamma = (\gamma_i)_{i=1}^N $ such that $\gamma_i \in \Pp(\Xx_i\times \Xx_i)$ has marginals $\mu^0_i$ and $\mu^1_i$.
These transport plans define interpolations between $\mu^0_i$ and $\mu^1_i$, defined for $t\in [0,1]$ as
\begin{align}\label{eq:mu-t}
\mu^t_i = ((1-t)\pi^1 + t\pi^2)_\# \gamma_i
\end{align}
where $\pi^1$ (resp.~$\pi^2$) is the projection on the first (resp.~second) factor of $\Xx_i\times \Xx_i$. In other terms, $\mu^t_i$ is characterized by
\[
\int_{\Xx_i} \varphi_i (x_i) \d\mu^t_i(x_i) =  \int_{\Xx_i\times\Xx_i} \varphi_i((1-t)x_i+ty_i)\d\gamma_i(x_i,y_i),
\qquad \forall \varphi_i \in \Cc(\Xx_i).
\]
Our main result  is as follows.

\begin{theorem}\label{thm:main}
For $p,k\in\NN^*$, $p\leq k$, if $c\in \Cc^{k+p}(\Xx)$ then the parametrized Schr\"odinger map  $t\mapsto \bphi^t\coloneqq S(\bmu^t)$ belongs to $\Cc^{p}([0,1];\tilde \Cc^k)$. Moreover, there exists $C>0$ that only depends on $\Vert c\Vert_{\Cc^{k+1}}$ and $N$ such that
\[
\Vert \bphi^t-\bphi^s\Vert_{\tilde \Cc^k} \leq C \vert t-s\vert \sqrt{\cost(\gamma)}
\]
where $\cost(\gamma)\coloneqq \sum_{i=1}^N \int_{\Xx_i\times \Xx_i} \Vert y_i-x_i\Vert^2\d\gamma_i(x_i,y_i)$ is the $L^2$-transport cost associated with $\gamma$.
\end{theorem}
\
The detailed proof is postponed to the next section, the basic ingredient being the application of the Implicit Function Theorem to the map $G(\bphi, t):=T(\bphi, \bmu^t)$. We can make the following comments:
\begin{itemize}

\item Tracking the constants in the proof, it can be seen that $C$ depends exponentially on the oscillation of the cost $\sup_x c(x)-\inf_x c(x)$ and polynomially on $\Vert c\Vert_{\Cc^{k+1}}$.
\item From the primal-dual relation Eq.~\eqref{eq:opt-primal-dual}, one could easily deduce stability results for the primal variable $\gamma$ from this theorem. 
\item The fact that the map $t\mapsto S(\bmu^t)$ belongs to $\Cc^p([0,1];\tilde \Cc^k)$, also holds if $(\bmu^t)_{t\in [0,1]}$, instead of being of the form \eqref{eq:mu-t}, is of the form $\mu_i^t=\xi_i(\cdot,t)_\#\mu_i^0$ for $i\in [N]$, for some $\mu_i^0\in \Pp(\Xx_i)$ and a measurable $\xi_i(x_i,\cdot)\in \Cc^p([0,1];\Xx_i)$ with a $\Cc^p$ norm uniformly bounded in $x_i$. This can be seen by suitably adapting the proof of Lem~\ref{lem:regularity-of-G}.
\end{itemize}

Applying Thm.~\ref{thm:main} by choosing $\gamma_i$ as the optimal transport plan between $\mu_i$ and $\nu_i$ immediately leads to the following Lipschitz continuity result for the Schr\"odinger map.

\begin{corollary}\label{cor:lipschitz-multi}
For $k\in \NN^*$, assume that $c\in \Cc^{k+1}(\Xx)$.
The Schr\"odinger map $S:\prod_{i=1}^N  \Pp(\Xx_i) \to \tilde \Cc^k$ is Lipschitz continuous, i.e.~there exists $C>0$ such that, for all $\bmu,\bnu \in \prod_{i=1}^N\Pp(\Xx_i)$, letting $(\bphi^{\bmu},\bphi^{\bnu})=(S(\bmu),S(\bnu))$,
\[
\Vert \bphi^{\bmu} - \bphi^{\bnu}\Vert_{\tilde \Cc^k} \leq C \bW(\bmu,\bnu).
\]
\end{corollary}

\rev{Our approach will also enable us to deduce a control of the  Schr\"odinger potentials and their derivatives in terms of negative Sobolev distances between the marginals  (see paragraph \ref{proofsobneg} for detailed definitions):

\begin{proposition}\label{cor:lipschitz-sobolev-multi}
Assume that $c\in \Cc^{k+p}(\Xx)$ with $p>d/2$ and $p\in \mathbb{N}^*$.
The Schr\"odinger map $S:\prod_{i=1}^N  \Pp(\Xx_i) \to \tilde \Cc^k$ is Lipschitz continuous in the negative Sobolev norm $H^{-p}$, i.e.~there exists $C>0$ such that, for all $\bmu,\bnu \in \prod_{i=1}^N\Pp(\Xx_i)$, letting $(\bphi^{\bmu},\bphi^{\bnu})=(S(\bmu),S(\bnu))$,
\[
\Vert \bphi^{\bmu} - \bphi^{\bnu}\Vert_{\tilde \Cc^k} \leq C \Vert \bmu-\bnu \Vert_{H^{-p}}.
\]
\end{proposition}

Note that when $p>d/2+1$, by Morrey's Theorem (see Section~\ref{proofsobneg}) and our compactness assumption, there exists $C>0$ that only depends on $\Xx$ such that $\Vert \bmu-\bnu \Vert_{H^{-p}}\leq  C\cdot \boldsymbol{W_1}(\bmu,\bnu) \leq C\cdot \bW(\bmu,\bnu)$. Thus, the conclusion of Prop.~\ref{cor:lipschitz-sobolev-multi} is generally stronger than that of Cor.~\ref{cor:lipschitz-multi}, but this is at the expense of requiring more regularity on the cost function. 

For illustration purposes, let us explain how this inequality leads to nonasymptotic estimation guarantees for the Schr\"odinger potentials given random samples. In the two marginal case, this is essentially a known result, obtained via different means in~\cite{del2022improved,rigollet2022sample}. Specifically, suppose that $\hat \bmu$ is an empirical measure built by drawing $n$ independent samples from each of the measures $\bmu_i$, $i\in N$. Then, since $H^{p}$ is a Reproducible Kernel Hilbert Space with a bounded kernel for $p>d/2$ (by Morrey's Theorem again), Hoeffding's inequality shows that $\Vert \hat \bmu-\bmu \Vert_{H^{-p}}$ is bounded by $C\cdot n^{-1/2}\sqrt{\log(1/\delta)}$ with probability $1-\delta$ where here $C$ depends only on $\Xx$. By Prop.~\ref{cor:lipschitz-sobolev-multi}, this directly translates into a high-probability bound on $\Vert \bphi^{\bmu} - \bphi^{\hat \bmu}\Vert_{\tilde \Cc^k}$. It might also be worth mentioning that Proposition \ref{cor:lipschitz-sobolev-multi} is obtained by linearly interpolating the marginals, as such, it is a slight departure from the rest of the paper which essentially focuses on displacement interpolation.
}

\section{Proofs}\label{sec:proofs}
The main tool to prove Thm.~\ref{thm:main} is the Implicit Function Theorem. We will apply it to the function $G:\tilde \Cc^k \times [0,1]\to  \Cc^k$ defined as
\begin{align}\label{def: G}
G(\bphi,t) &\coloneqq T(\bphi,\bmu^t)
\end{align}
whose expression\footnote{the maps $G$ and $T$ take values in $\Cc^k$ but it will sometimes be convenient to compose them from the left with the canonical projection $\Cc^k \to \tilde \Cc^k$, slightly abusing notations, we will still denote by $G$ and $T$ these maps with values in $\tilde \Cc^k$.} is, using the convention $y_i\coloneqq x_i$,
\begin{align*}
G_i(\bphi,t)(x_i) = \phi_i(x_i)+\log \Big( \int_{\Xx_{-i} \times \Xx_{-i}} e^{\sum_{j\neq i} \phi_j((1-t)x_j+ty_j)-c((1-t)x+ty)} \d\gamma_{-i}(x_{-i},y_{-i})\Big).
\end{align*}
For this purpose, in the next sections, we study the \rev{properties of the maps $T$ and $G$}.

\subsection{Invertibility of the differential of $T$}
Let us fix $\bmu= ( \mu_1, \dots, \mu_N) \in \prod_{i=1}^N \Pp(\Xx_i)$ and $\mu=\otimes_{i=1}^N \mu_i\in \Pp(\Xx)$ and study the map $\bphi\mapsto T(\bphi,\bmu)$, which is a self-map of $\tilde \Cc^k$.
Note that $T$ is of class $C^\infty$ in the first variable and its differential is given, for $\bh\in \tilde \Cc^k$, by
\begin{align*}
D_\phi T_i(\bphi,\bmu)(\bh)(x_i) 
&=h_i(x_i) + \int_{\Xx_{-i}}\big(\sum_{j\neq i} h_j(x_j)\big)q_{-i}(x_{-i}|x_i)\d \mu_{-i}(x_{-i}).
\end{align*}
where we have introduced the function $q_{-i}$ defined, with the convention $x'_i=x_i$, by
\begin{align*}
q_{-i}(x_{-i}|x_i) &\coloneqq  \frac{e^{\sum_{j\neq i} \phi_j(x_j)-c(x)}}{\int_{\Xx_{-i}} e^{\sum_{j\neq i} \phi_j(x'_j)-c(x')}\d\mu_{-i}(x'_{-i})}.
\end{align*}
Note that $q_{-i}$ depends on $\bphi$ and $\bmu$ although this is not explicit in the notation. Similarly, let
 \begin{align}\label{eq:def-q}
 q (x) & \coloneqq \frac{e^{\sum_j \phi_j(x_j)-c(x)}}{\int_{\Xx} e^{\sum_{j} \phi_{j}(x_{j}')-c(x')}\d\mu(x')}, 
 & q_i (x_i) &\coloneqq \frac{\int_{\Xx_{-i}} e^{\sum_{j} \phi_j(x_j)-c(x)}\d\mu_{-i}(x_{-i})}{\int_{\Xx} e^{\sum_{j} \phi_{j}(x'_{j})-c(x')}\d\mu(x') }.
 \end{align}
Observe that if $\bphi$ and $c$ are of class $\Cc^k$ then the functions $q,q_i,q_{-i}$ are of class $\Cc^k$ as well. These functions are densities of probability densities in the sense that it holds
 \begin{align}\label{eq:def-Q}
 Q &\coloneqq q \mu \in \Pp(\Xx),& Q_i&\coloneqq q_i\mu_i \in \Pp(\Xx_i),& Q_{-i}(\cdot|x_i)&\coloneqq q_{-i}(\cdot|x_i)\mu_{-i} \in \Pp(\Xx_{-i}), \; \forall x_i\in \Xx_i.
 \end{align}
By construction,   for each $i$,  $Q_i$ is the $i$-th marginal of $Q$ on $\Xx_i$ and  $Q_{-i}$ is the disintegration of $Q$ with respect to this marginal, i.e.:
\begin{align}\label{eq:desint-Q}
 \d Q ( x_i ,\;  x_{-i}) =  \d Q_{-i} ( x_{-i} \vert x_i )  \d Q_i( x_i). 
 \end{align}
In the next lemma, we remark that these densities are uniformly bounded from above and below by positive quantities, a fact which we will often use in the following. 
\begin{lemma}\label{lem:bound-dQ/dmu}
Let $q$ be defined by \eqref{eq:def-q}. Then, for all $x \in \Xx$,
\[
e^{-2(N \Vert \bphi \Vert_{\tilde \Cc^0} + \Vert c \Vert_{\Cc^0})} \leq q(x) \leq e^{2(N \Vert \bphi \Vert_{\tilde \Cc^0} + \Vert c \Vert_{\Cc^0})}. 
\]
Moreover $q_i$ and $q_{-i}$ satisfy the same bounds.
\end{lemma}
\begin{proof}
From the definition of $q$, for all $x \in \Xx$,
\begin{align*}
    q(x) & = \frac{e^{\sum_j \phi_j(x_j)-c(x)}}{\int_{\Xx} e^{\sum_{j} \phi_{j}(x_{j}')-c(x')}\d\mu(x')} \leq \frac{e^{\sum_j \Vert \phi_j\Vert_\infty + \Vert c \Vert_{\Cc^0}}}{e^{-\sum_j \Vert \phi_j\Vert_\infty - \Vert c \Vert_{\Cc^0}}} \leq e^{2(N \Vert \bphi \Vert_{ \Cc^0} + \Vert c \Vert_{\Cc^0})},
\end{align*}
since $\mu \in \Pp(\Xx)$. In addition, from the definition of $q$, we remark that $q$ does not depend of the representative of $\bphi$ in $\tilde \Cc^0$ which gives the upper bound on $q$. We obtain the lower bound, as well as the result for $q_i$ and $q_{-i}$ with the same arguments.
\end{proof}

Let us also remark that in the previous bounds, the norm $\Vert c\Vert_{\Cc^0}$ can be replaced by $\inf_{\kappa \in \RR} \Vert c+\kappa\Vert_{\Cc^0}=(\sup c - \inf c)/2$, i.e.~half the \emph{oscillation of $c$} (in fact, the Schr\"odinger map is invariant if $c$ changes by an additive constant). \rev{The following lemma is central in our development and is an adaptation of~\cite[Prop.~3.1]{carlier2020differential} (with different functional spaces). The first claim of invertibility appeared in a similar form in~\cite[Lem.~5]{gonzalez2022weak} where it is key to prove a central limit theorem for EOT, but our proof (Step 1) is different as (i) in our context there is no natural way to get rid of the non-uniqueness of Schr\"odinger potentials so we work directly in the quotient space $\tilde \Cc^k$ and (ii) our approach leads to control on the norm on the inverse (the second part of the claim).}

\begin{lemma}\label{lem:inverse-function}
Let $k\in \NN^*$ and assume that  $\bphi \in \tilde \Cc^0$ and $c\in \Cc^k(\Xx)$.

Then $D_{\bphi} T(\bphi,\bmu)$ is an invertible linear self-map of $\tilde \Cc^k$. Moreover, there exists $C>0$ that only depends on $N$, $\Vert c\Vert_{\Cc^k}$ and $\Vert \bphi\Vert_{\tilde \Cc^0}$ such that 
\[
\Vert [D_{\bphi} T(\bphi,\bmu)]^{-1}\Vert_{\tilde \Cc^k\to \tilde \Cc^k}\leq C.
\]
\end{lemma}

\begin{proof}
We have $D_{\bphi} T(\bphi,\bmu)=\Id +L$ with 
 \begin{align*}
 L_{i}(\bh)(x_i) 
 &=\int_{\Xx_{-i}} \Big( \sum_{j\neq i} h_j(x_j) \Big) q_{-i}(x_{-i}|x_i) \d\mu_{-i}(x_{-i}).
 \end{align*}
Observe that since $c\in \Cc^k(\Xx)$, $L_i(\bh) \in \Cc^k(\Xx_i)$ with its derivatives up to order $k$  equi-continuous when $h$ runs through a bounded set of $\tilde \Cc^0$.  It follows, by Arzel\`a-Ascoli Theorem, that $L:\tilde \Cc^{0} \to \tilde \Cc^{k}$ is compact, and a fortiori $L:\tilde \Cc^{k} \to \tilde \Cc^{k}$ is compact too.

\emph{Step 1.} Let us show that $\id +L$ is invertible.
 Let $\bh \in \prod \Cc(\Xx_i)$ be such that $\bh+L(\bh)=0$ in $\tilde \Cc^{0}$, i.e. 
  \begin{align}\label{eq:kernelquotient}
h_i(.) + \int_{\Xx_{-i}} \Big( \sum_{j\neq i} h_j(x_j) \Big) \d Q_{-i}(x_{-i} \vert .)=\lambda_i, \; i=1, \dots, N, \; \sum_{j=1}^N \lambda_j=0.
\end{align}
 Integrating \eqref{eq:kernelquotient} with respect to $Q_i$, we deduce from \eqref{eq:desint-Q}, that 
 \[\sum_{k=1}^N \int_{\Xx_k} h_k \d Q_k=\lambda_i, \; \; i=1, \dots, N\]
 so that all the $\lambda_i$'s are equal to $0$; hence $\bh+L(\bh)=0$ in $\Cc^0$ (and not only in the quotient $\tilde \Cc^{0}$). Then, taking the dot product of $\bh$ with $\bh+L(\bh)$ in $\prod L^2(Q_i)$ and using \eqref{eq:desint-Q}, it follows, reasoning as in~\cite{carlier2020differential},
 \begin{align*}
0 &= \sum_{i=1}^N \int_{\Xx_i} h_i(x_i) \Big(h_i(x_i)+\int_{\Xx_{-i}} \Big( \sum_{j\neq i} h_j(x_j) \Big) \d Q_{-i}(x_{-i}|x_i)\Big)\d Q_{i}(x_i)\\
 & = \sum_{i} \int_{\Xx_i} h_i(x_i)^2 \d Q_i(x_i)+ \sum_{i\neq j}  \int_\Xx h_i(x_i)h_j(x_j)\d Q(x)\\
 &= \int_\Xx \Big(\sum_{i} h_i(x_i)\Big)^2\d Q(x).
 \end{align*}
  We deduce that $x\mapsto \sum h_i(x_i)$ is equal to  $0$ as a function in $L^2(Q)$ and hence in $L^2(\mu)$.
 
 Now consider the space $\tilde L^2_{\bmu} \coloneqq \prod_{i=1}^N L^2(\mu_i)/\sim$ which, endowed with the quotient space structure, is also a Hilbert space. By Lem.~\ref{lem:norms} (proved hereafter), it follows that $\bh \sim 0$ in $\tilde L^2_{\bmu}$, i.e. there exists $\kappa \in \RR^n$ such that $\sum \kappa_i=0$ and $h_i(x_i)= \kappa_i$ for $\mu_i$-a.e.~$x_i$. It only remains to show that this equality holds in fact everywhere. Using $\bh=-L(\bh)$ and $q_{-i}\mu_{-i}\in \Pp(\Xx_{-i})$, it holds for $x_i\in \Xx_i$
 \begin{align*}
 h_i(x_i) &= - \int_{\Xx_{-i}} \Big( \sum_{j\neq i} h_j(x_j)\Big) q_{-i}(x_{-i}|x_i)\d \mu_{-i}(x_{-i}) = -\sum_{j\neq i} \kappa_j = \kappa_i
 \end{align*}
 Thus $\bh = 0$ in $\tilde \Cc^k$. 
 Conversely, any $\bh = 0$ in $\tilde \Cc^k$  also clearly belongs to $\ker(D_{\bphi} T(\bphi,\bmu))$. Hence $\ker(D_{\bphi} T(\bphi,\bmu))$ is precisely the equivalence class of $0$  i.e.~$D_{\bphi} T(\bphi,\bmu)$ is injective on $\tilde \Cc^k$. Since $L$ is a compact operator of $\tilde \Cc^k$, it follows from the Fredholm Alternative Theorem~\cite[Chap.~6]{brezis2011functional} that the range of $\Id+L$ is $\tilde \Cc^k$. Hence $D_{\bphi} T(\bphi,\bmu)$ is onto and therefore an invertible linear self-map of $\tilde \Cc^k$.

\emph{Step 2.} Now let us estimate the operator norm of $D_{\bphi} T(\bphi,\bmu)^{-1}$ as a self-map of $\tilde \Cc^k$. Let $\bh\in \tilde \Cc^k, \bg\in \tilde \Cc^k$ be such that $\bh+L(\bh)=\bg$. Let us choose the representative of $\bg$ that satisfies 
 \begin{align}\label{eq:output-invariance}
 \int_{\Xx_1} g_1(x_1)\d Q_1(x_1)=\dots = \int_{\Xx_N} g_N(x_N)\d Q_N(x_N).
 \end{align}
 Reasoning as above, it holds
 \begin{align*}
 \sum_{i=1}^N \int_{\Xx_i} g_i(x_i)h_i(x_i)\d Q_{i}(x_i) 
 & = \sum_{i} \int_{\Xx_i} h_i(x_i)^2 \d Q_i (x_i)+ \sum_{i,j,\, i\neq j}  \int_\Xx h_i(x_i)h_j(x_j)\d Q (x)\\
 & = \int_\Xx \Big(\sum_{i} h_i(x_i)\Big)^2\d Q (x)
 \end{align*}
 where the first integral is unambiguously defined because, thanks to our choice of representative for $g$, it does not depend on the representative chosen for $h$. Let us choose the optimal representative for $\bh$ in $\tilde L^2_{\bmu}$ appearing in Lem.~\ref{lem:norms}.
We have, for some $C>0$ that may change from a line to another but only depends on $N$, $\Vert c\Vert_{\Cc_0}$ and $\Vert \bphi\Vert_{\tilde \Cc_0}$:
 \begin{align*}
\Vert \bh \Vert^2_{\prod L^2(\mu_i)}
&\overset{(i)}{\leq}
 N \left\Vert \oplus_{i=1}^N h_i\right\Vert^2_{L^2(\mu)} \\
& \overset{(ii)}{\leq}
 C \left\Vert \oplus_{i=1}^N h_i\right\Vert^2_{L^2(Q)} \\ 
& \overset{(iii)}{\leq}
  C \Vert \bg\Vert_{\prod L^2(Q_i)} \Vert \bh\Vert_{\prod L^2(Q_i)}\\
  & \overset{(iv)}{\leq} C\Vert \bg\Vert_{\prod L^2(Q_i)} \Vert \bh\Vert_{\prod L^2(\mu_i)}
 \end{align*}
 where we have used (i) Lem.~\ref{lem:norms} (where the notation $\oplus$ is defined) , (ii)\&(iv) Lem.~\ref{lem:bound-dQ/dmu}, and (iii) the previous computation and Cauchy-Schwarz inequality  in $\prod L^2(Q_i)$.
It follows, invoking once again Lem.~\ref{lem:norms} and Lem.~\ref{lem:bound-dQ/dmu}, that 
\[
\Vert  \bh\Vert_{\tilde L^2_{\bmu}} = \Vert \bh \Vert_{\prod L^2(\mu_i)}  \leq C \Vert \bg\Vert_{\prod L^2(Q_i)} = C \Vert \bg\Vert_{ \tilde L^2_Q} \leq C \Vert \bg\Vert_{ \tilde \Cc^0}.
\]
For the last equality, we have used the fact that the $\tilde L^2_Q$ norm is precisely the $\prod L^2(Q_i)$ norm of the representative that satisfies Eq.~\eqref{eq:output-invariance}, by Lem.~\ref{lem:norms} (here $\tilde L^2_Q$ is defined similarly as $\tilde L^2_{\bmu}$ from the marginals $Q_i$ of $Q$).

\emph{Step 3.} We now improve the $\tilde L^2_{\bmu}$ control into a $\tilde \Cc^k$ control. Restarting from $\bh+L(\bh)=\bg$, it holds
\begin{align}\label{hLhg}
 h_i(x_i) = g_i(x_i) - \int_{\Xx_{-i}} \Big( \sum_{j\neq i}h_j(x_j) \Big) q_{-i}(x_{-i}|x_i)\d\mu_{-i}(x_{-i}).
\end{align}
Thanks to our control on $\Vert \bh\Vert_{\prod L^2(\mu_i)}$ by  $\Vert \bg\Vert_{\tilde \Cc^0}$, given constants $\kappa_i$, it follows from \eqref{hLhg} that
\[ \Vert h_i + \kappa_i \Vert_{\tilde \Cc^0(\Xx_i)} \leq \Vert g_i + \kappa_i \Vert_{ \tilde \Cc^0(\Xx_i)}+ C \Vert \bg\Vert_{ \tilde \Cc^0} \]
summing over $i$ and minimizing with respect to the $\kappa_i$'s summing to $0$, we get
\[ \Vert \bh  \Vert_{\tilde \Cc^0} \leq  C \Vert \bg\Vert_{ \tilde \Cc^0}.\]
In a similar way, using the fact that $c\in \Cc^k$, successive differentiations of \eqref{hLhg} yield 
\[\Vert \bh \Vert_{\tilde \Cc^k}= \Vert [D_{\bphi} T(\bphi,\bmu)]^{-1}(\bg) \Vert_{\tilde \Cc^k}  \leq C   \Vert \bg \Vert_{\tilde \Cc^k}
\]
for a constant $C$ that only depends on $N$, $\Vert c\Vert_{\Cc^k}$ and $\Vert \bphi\Vert_{\tilde \Cc^0}$.
\end{proof}

To end this section, we prove Lem.~\ref{lem:norms} used in the previous proof. 

\begin{lemma}\label{lem:norms}
For $\bh\in\prod_{i=1}^N L^2(\mu_i)$, denoting $ \oplus_{i=1}^N h_i : x\mapsto \sum_{i=1}^N h_i(x_i)$ it holds
\[
 \Vert \bh \Vert^2_{\tilde L^2_{\bmu}}\leq \left\Vert \oplus_{i=1}^N h_i\right\Vert^2_{L^2(\mu)} \leq N \Vert \bh \Vert^2_{\tilde L^2_{\bmu}} .
\]
Moreover, the quotient norm is achieved by the unique representative $\hat \bh \sim \bh$ that satisfies $\int_{\Xx_1} \hat h_1 \d\mu_1 = \dots = \int_{\Xx_N} \hat h_N \d\mu_N$, i.e. it holds $\Vert h\Vert^2_{\tilde L^2_{\bmu}} =\sum_{i=1}^N \Vert \hat h_i\Vert^2_{L^2(\mu_i)}$. 
\end{lemma}

\begin{proof}
By definition,
\begin{align}\label{eq:optim-norm}
\Vert \bh \Vert^2_{\tilde L^2_{\bmu}} &=\min_{\substack{\kappa\in \RR^N\\ \sum_i\kappa_i=0}} \sum_{i=1}^N \int_{\Xx_i} (h_i(x_i)-\kappa_i)^2\d\mu_i(x_i).
\end{align}
A vector $\kappa\in \RR^N$ solves this problem iff $\sum_i \kappa_i=0$ and there exists a Lagrange multiplier $\nu\in \RR$ such that for $i\in [N]$,
\[
0=\int_{\Xx_i} (h_i(x_i)-\kappa_i)\d\mu_i(x_i) -\nu  = \E_{\mu_i}[h_i] -\kappa_i -\nu.
\]
with the shorthand $\E_\mu[h]\coloneqq \int h\d\mu$ and $\Var_{\mu}(h)\coloneqq \E_\mu[(h-\E_\mu[h])^2]$. It follows that $\nu = \frac1N \sum_{i=1}^N \E_{\mu_i}[h_i] $ and as a consequence
\begin{align*}
\Vert \bh \Vert^2_{\tilde L^2_{\bmu}} &=\sum_{i=1}^N \int_{\Xx_i} (h_i(x_i)-\E_{\mu_i}[h_i]+\nu)^2\d\mu_i(x_i)\\
&= \sum_{i=1}^N\Big[ \int_{\Xx_i} (h_i(x_i)-\E_{\mu_i}[h_i])^2\d\mu_i(x_i) + \nu^2\Big]\\
&= \sum_{i=1}^N \Var_{\mu_i}(h_i) + \frac1N \Big(\sum_{i=1}^N \int h_i\d\mu_i\Big)^2
\end{align*}
where the second equality follows by expanding the square and observing that the cross-terms vanish. On the other hand, using the fact that $\sum_{i} \kappa_i=0$,  it holds
\begin{align*}
\left\Vert \oplus_{i=1}^N h_i\right\Vert^2_{L^2(\mu)} 
&= \int_{\Xx} \big(\sum_{i=1}^N (h_i(x_i)-\kappa_i)\big)^2\d\mu(x)\\
&= \sum_{i,j, i\neq j} \int_{\Xx_i \times \Xx_j} (h_i(x_i) - \E_{\mu_i}[h_i]+\nu)(h_j(x_j) - \E_{\mu_j}[h_j]+\nu)\d\mu_i(x_i)\d\mu_j(x_j) \\
&\quad+ \sum_{i} \int_{\Xx_i} (h_i(x_i) - \E_{\mu_i}[h_i])+\nu)^2\d\mu_i(x_i)\\
&=N(N-1)\nu^2 + \Vert \bh \Vert^2_{\tilde L^2_{\bmu}}\\
& = \sum_{i=1}^N \Var_{\mu_i}(h_i) + \Big(\sum_{i=1}^N \int_{\Xx_i} h_i\d\mu_i\Big)^2
\end{align*}
The first claim follows. For the second claim, observe that this representative $\hat h$ satisfies the optimality condition of Eq.~\eqref{eq:optim-norm}.
\end{proof}

\subsection{Differentiability of $G$}

Let us first establish the regularity of the map $G$.
\begin{lemma}\label{lem:regularity-of-G}
For $p,k\in \NN^*$, $p\leq k$, if $c\in \Cc^{k+p}(\Xx)$, then the map $G: \Cc^k \times [0,1]\to  \Cc^k$ is of class $\Cc^{p}$.
\end{lemma}

\begin{proof}
The $i$-th component of  $G$ can be expressed as
\[
G_i(\bphi,t)(x_i) = \phi_i(x_i)+\log \Big( \int_{\Xx_{-i} \times \Xx_{-i}} e^{\sum_{j\neq i} \phi_j((1-t)x_j+ty_j)-c(x_i, (1-t)x_{-i}+ty_{-i})} \d\gamma_{-i}(x_{-i},y_{-i})\Big)
\]
Fixing $i$ and  $(x_{-i}, y_{-i})\in \Xx_{-i} \times \Xx_{-i}$, let us observe that when  $c\in \Cc^{k+p}(\Xx)$,  the curve $t\in [0,1] \mapsto c(., (1-t)x_{-i}+ty_{-i} ) \in \Cc^k(\Xx_i)$ is of class $\Cc^p$ and that its  derivatives up to order $p$ can be bounded independently of $(x_{-i}, y_{-i})$. Now, for $j\neq i$ (and fixed $x_j$ and $y_j$ in $\Xx_j$), consider the  real-valued map $L_j:(\phi_j, t)\in \Cc^k(\Xx_j)\times [0,1]\mapsto \phi_j(x_j +t(y_j-x_j))$. For $k=1$, this map admits partial derivatives with respect to $t$ and $\phi_j$ which are given respectively by $\nabla \phi_j(x_j+t(y_j-x_j))^\top (y_j-x_j)$ and $L_j(., t)$, both being continuous (for the $\Cc^1$ norm for $\phi_j$) so that $L_j \in \Cc^1(\Cc^1(\Xx_j)\times [0,1], \RR)$ note also that the first-order partial derivatives of $L_j$ can be bounded by a constant depending on the $\Cc^1$ norm of $\phi_j$ but not on $x_j, y_j$. 

For $k\geq 2$, we can argue inductively. Indeed, by the previous argument,  showing $k$ times continuous differentiability of $L_j$ amounts to showing $k-1$ times continuous differentiability of $L_j$ applied to  $\nabla \phi_j(.)^\top (y_j-x_j)$ and $t$. This shows that $L_j \in \Cc^k(\Cc^k(\Xx_j)\times [0,1], \RR)$, with bounds on derivatives up to order $k$ controlled by the $\Cc^k$ norm of $\phi_j$  independently of $(x_j, y_j)\in \Xx_j\times \Xx_j$. By the chain rule and differentiating under the integral sign by dominated convergence, we can readily conclude that $G$ is of class $\Cc^{\min\{k,p\}}=\Cc^p$ from $\Cc^k(\Xx)\times [0,1]$ to $\Cc^k(\Xx)$.
\end{proof}

We now give a quantitative regularity estimate for the partial derivative of $G$ in its real variable $t$.

\begin{lemma}\label{lem:differential-in-t}
Let $k\in \NN^*$ and assume that $c\in \Cc^{k+1}(\Xx)$. Given $\bphi \in \tilde \Cc^1$, the partial differential of $G$ in $t$ satisfies
$$
\Vert D_t G(\bphi,t) \Vert_{\tilde \Cc^{k}} \leq C \sqrt{\cost(\gamma)}
$$
where $C>0$ only depends on $\Vert \bphi\Vert_{\tilde \Cc^1}$ and $\Vert c\Vert_{\Cc^{k+1}}$ and $\cost(\gamma)$ is the transport cost associated with $\gamma$ as in Thm.~\ref{thm:main}.
\end{lemma}
\begin{proof}

By Lem.~\ref{lem:regularity-of-G}, $G$ is differentiable in $t$.
Using the shorthand $x^t \coloneqq (1-t)x+ty$ and again the convention $y_i=x_i$, it holds
\begin{align}\label{eq:dtG}
\frac{d}{dt}G_i(\bphi,t)(x_i) =\int_{\Xx_{-i} \times \Xx_{-i}} \Big(\sum_{j\neq i} (y_j-x_j)^\top (\nabla \phi_j(x_j^t) - \nabla_j c(x^t))\Big)\d Q^t_{-i}(x_{-i},y_{-i}|x_i)
\end{align}
with $Q^t_{-i} \coloneqq q^t_{-i}\gamma_{-i} \in \prod_{j\neq i} \Pp(\Xx_j\times \Xx_j)$ and, posing $(x')^t_i=x_i^t$,
$$
q^t_{-i}(x_{-i},y_{-i}|x_i) \coloneqq  \frac{e^{\sum_{j\neq i} \phi_j(x^t_j)-c(x^t)}}{\int_{\Xx_{-i} \times \Xx_{-i}} e^{\sum_{j\neq i} \phi_j((x')^t_j)-c((x')^t)}\d\gamma_{-i}(x'_{-i},y'_{-i})}.
$$
Reasoning as in Lem.~\ref{lem:bound-dQ/dmu}, this function $q^t_{-i}$ admits positive upper and lower bounds only depending on $\Vert \bphi\Vert_{\tilde \Cc_0}$ and $\Vert c\Vert_{\Cc^0}$.
Let us now control Eq.~\eqref{eq:dtG}, starting with a control in uniform norm. First, by Cauchy-Schwarz in $L^2(Q^t_{-i}(\cdot,\cdot|x_i))$, for $i\in [N]$ and $x_i\in \Xx_i$,
\begin{multline}
\label{eq:CSG}
\Big\vert \frac{d}{dt}G_i(\bphi,t)(x_i) \Big\vert^2
\\ \leq \left( \int_{\Xx_{-i} \times \Xx_{-i}}\!\! \Vert y-x\Vert^2 \d Q^t_{-i}(x_{-i},y_{-i}|x_i) \right) \left( \int_{\Xx_{-i} \times \Xx_{-i}}\!\! \Vert (\nabla \bphi - \nabla c)(x^t)\Vert^2 \d Q^t_{-i}(x_{-i},y_{-i}|x_i)\right) 
\end{multline}
where $\nabla \bphi \coloneqq (\nabla\phi_1,\dots,\nabla \phi_N)$. Observe that the second factor is uniformly bounded for $x_i\in \Xx_i$ because $Q^t_{-i}(\cdot,\cdot|x_i)$ is a probability measure and both $\bphi$ and $c$ are continuously differentiable on a compact set. It  follows,
\begin{align*}
\Big\Vert \frac{d}{dt}G_i(\bphi,t) \Big\Vert_{\Cc^0} &\leq C \sup_{x_i\in \Xx_i}\left( \int_{\Xx_{-i} \times \Xx_{-i}} \Vert y-x\Vert^2 \d Q^t_{-i}(x_{-i},y_{-i}|x_i) \right)^{1/2} \\
&\leq C' \left( \int_{\Xx_{-i} \times \Xx_{-i}} \Vert y-x\Vert^2 \d \gamma_{-i}(x_{-i},y_{-i}) \right)^{1/2} \\
&\leq  C'\sqrt{\cost(\gamma)}
\end{align*}
where $C,C'$ depend on $\Vert c\Vert_{\Cc^1} $ and $\Vert \bphi \Vert_{\tilde \Cc^1}$ only. Moreover, one can further differentiate Eq.~\eqref{eq:dtG} in $x_i$ and obtain analogous bounds because this variable only appears in the term $\nabla_j c$ which is of regularity $\Cc^{k}$ and in the factor $q^t_{-i}$ which is of regularity $\Cc^{k+1}$. With this reasoning, it follows
$$
\Big\Vert \frac{d}{dt}G_i(\bphi,t) \Big\Vert_{\tilde \Cc^k} \leq C_k\sqrt{\cost(\gamma)}.
$$
where $C_k$ depends on $\Vert c\Vert_{\tilde \Cc^{k+1}} $ and $\Vert \bphi \Vert_{\tilde \Cc^1}$ only.

\end{proof}

\subsection{Proof of Thm.~\ref{thm:main}}\label{proofLipW2}

\begin{proof}
Let us first observe that $c$ can be extended in a $\Cc^{k+p}$ way to a convex open set containing $\Xx$. One can therefore extend by extrapolation the definition of $\bmu^t$ to an open time interval $(-\varepsilon, 1+\varepsilon)$, for some $\varepsilon>0$ containing $[0,1]$.
We shall then apply the Implicit Function Theorem (IFT) to $G:\tilde \Cc^k \times [0,1]\to \tilde \Cc^k$ defined in \eqref{def: G}. 

The existence and continuity on $[0,1]$ of the Schr\"odinger map $t\mapsto \bphi^t$ is guaranteed by Prop.~\ref{prop:preliminaries}. 

In Lem.~\ref{lem:regularity-of-G}, we have shown that $G$ is of class $\Cc^p$ and in Lem.~\ref{lem:inverse-function}, we have shown that $D_{\bphi} G(\bphi^t,t) = D_{\bphi} T(\bphi^t,\bmu^t)$ is an invertible linear self-map of $\tilde \Cc^k$. Thus all the hypotheses are gathered to apply the Implicit Function Theorem, see e.g.~\cite[Thm.~10.2.1]{dieudonne2011foundations}: the map $t\mapsto \bphi^t$ is of class $\Cc^p$ on $[0,1]$ and its derivative is given by
\[
D_t\bphi^t = -  [D_{\bphi} T(\bphi^t,\bmu^t)]^{-1} \circ D_t G(\bphi^t,t).
\]
Moreover, we have by respectively Lem.~\ref{lem:inverse-function} and Lem.~\ref{lem:differential-in-t} that there exists $C>0$ only depending on $N$, $\Vert c\Vert_{\Cc^{k+1}}$ and $\Vert \bphi^t\Vert_{\tilde \Cc^1}$ such that 
\begin{align*}
\Vert [D_{\bphi} T(\bphi^t,\bmu^t)]^{-1}\Vert_{\mathrm{op}}\leq C &&\text{and}&&\Vert D_t G(\bphi^t,t)\Vert_{\mathrm{op}}\leq C \sqrt{\cost(\gamma)}.
\end{align*}
Since we know by Prop.~\ref{prop:preliminaries} that $\Vert \bphi^t\Vert_{\tilde \Cc^1}$ is a priori bounded by $\Vert c\Vert_{\Cc^1}$, it follows that $\Vert D_t\bphi^t\Vert \leq C' \sqrt{\cost(\gamma)}$ for some $C'_k>0$ that only depends on $\Vert c\Vert_{\Cc^{k+1}}$.  The Lipschitz estimate in Thm.~\ref{thm:main} follows by the mean-value inequality.
\end{proof}

\rev{\subsection{Proof of Proposition \ref{cor:lipschitz-sobolev-multi}}\label{proofsobneg}

Recall that the Sobolev space $H^p(\Xx_i)$ consists of all functions $f_i \in L^2(\Xx_i)$ whose partial derivatives up to order $p$ belong to $L^2(\Xx_i)$ which is a Hilbert space for the norm
\[ \Vert f_i  \Vert^2_{H^p(\Xx_i)}:=\sum_{ \alpha \; : \; \vert \alpha \vert \leq p} \int_{\Xx_i} \vert \partial^\alpha f_i \vert^2.\]
If $p>d/2$, by Morrey's Theorem (see \cite{brezis2011functional}), $H^p(\Xx_i)$ embeds continuously into the space of continuous functions, hence, by duality, measures belong to the dual space $H^{-p}(\Xx_i)$. We can therefore define
\[\Vert \bmu-\bnu\Vert_{H^{-p}}:=\sum_{i=1}^{N} \Vert \mu_i-\nu_i \Vert_{H^{-p}}\]
where 
\[\Vert \mu_i-\nu_i \Vert_{H^{-p}}:=\sup\Big\{ \int_{\Xx_i } f_i \d (\mu_i-\nu_i) \; : \; \ \Vert f_i \Vert_{H^p(\Xx_i)} \leq 1\Big\}.\]
To obtain the bound announced in Prop. \ref{cor:lipschitz-sobolev-multi}, we simply consider the linear interpolation between $\bmu$ and $\bnu$, $\bmu^t:=\bmu+ t(\bnu-\bmu)$ for $t\in [0,1]$ and $G(\bphi, t):=T(\bphi,\bmu^t)$ as well as $\bphi^t:=S(\bmu^t) \in \tilde \Cc^k$ i.e. $G(\bphi^t, t)=0$. Recall that $\bphi^t$ is bounded in $\Cc^k$ by a constant that only depends on $c$. The same holds for the operator norm of $[D_{\bphi} T(\bphi^t,\bmu^t)]^{-1}$ in $\Cc^k$ as well.  To conclude as before by the implicit function theorem, we have to  differentiate $G$ with respect to $t$ and bound the $\Cc^k$ norm of $D_t G(\phi^t, t)$ by a constant depending on $c$ times $\Vert \bmu- \bnu\Vert_{H^{-p}}$. To simplify notations, let us set
\[\xi_i(\bphi, t)(x_i):=\frac{1}{\int_{\Xx_{-i}} e^{-c(x_i, x_{-i})+ \sum_{j\neq i} \phi_j(x_j)} \d \mu_{-i}^t (x_{-i})} \]
and observe that $\xi_i(\bphi^t, t)(.)$ has uniformly bounded derivatives up to order $k$ (with  bounds that depend on $\Vert c\Vert_{\Cc^k}$ only). If $N=2$, we simply have
\[D_t G_1(\bphi, t)(x_1)= \xi_1(\bphi, t)(x_1) \int_{\Xx_{2}} e^{-c(x_1,x_2)+\phi_2 (x_2)} \d (\nu_2-\mu_2)(x_2). \]
By Leibniz formula to bound the $k$ first derivatives of $D_t G_1(\bphi^t, t)(x_1)$, we then just have to bound the $k$ first derivatives of $x_1 \in \Xx_1 \mapsto \int_{\Xx_{2}} e^{-c(x_1,x_2)+\phi_2^t (x_2)} \d (\nu_2-\mu_2)(x_2)$  which are obviously controlled by $\Vert \nu_2 -\mu_2\Vert_{H^{-p}}$ times the $\Cc^{k+p}$ norm of 
$(x_1, x_2) \mapsto e^{-c(x_1,x_2)+\phi_2^t (x_2)}$ which can in turn be bounded by a constant only depending on $\Vert c \Vert_{\Cc^{k+p}}$. Proceeding in the same way for $x_2 \mapsto D_t G_2(\bphi^t, t)(x_2)$ gives the desired result. The case $N\geq 3$ is slightly  more tedious to write, for a fixed pair of indices $i\neq j$, we denote by $\Xx_{-(i,j)}$ the cartesian product of all the $\Xx_l$ but $i$ and $j$,  and write $x\in \Xx$ as $x=(x_i, x_j, x_{-(i,j)}) \in \Xx_i \times \Xx_j \times \Xx_{-(i,j)}$, likewise we write $\mu_{-(i,j)}^t$ for the tensor product of $\mu_l^t$ for $l\neq i$, $l\neq j$. Doing so, we have 
\[D_t G_i(\bphi, t)(x_i)= \xi_i(\bphi, t)(x_i) \sum_{j\neq i} \int_{\Xx_j} h_{ij}(\bphi, t) (x_i, x_j ) \d (\nu_j-\mu_j)(x_j)\]
where
\[h_{ij}(\bphi, t) (x_i, x_j ):=\int_{\Xx_{-(i,j)}} e^{-c(x_i, x_j, x_{-(i,j)})+ \phi_j(x_j)+ \sum_{l\neq i, l\neq j} \phi_l(x_l)} \d\mu^{t}_{-(i,j)}(x_{-(i,j)}) \]
so that by the same arguments as before, if $\alpha\in \NN^d$ with $\vert \alpha \vert\leq k$, we have for a constant $C$ only depending on $N$ and the $\Cc^{k+p}$ norm of $c$, possibly varying from one line to another:
\[  \vert \partial^\alpha D_t G_i(\bphi^t, t)(x_i) \vert \leq C \sum_{j\neq i} \Vert \mu_j -\nu_j \Vert_{H^{-p}(\Xx_j)} \Vert \partial^\alpha_{x_i} h_{ij}(\bphi^t, t)(x_i, .) \Vert_{H^p(\Xx_j)} 
\leq C \Vert \bmu-\bnu \Vert_{H^{-p}}.
\]
This enables us to conclude exactly as in the end of paragraph \ref{proofLipW2}.

}

\section{Smoothness of entropic optimal transport and Wasserstein gradient flows}\label{sec:applications}
In this section, we apply our main stability results to the analysis of Wasserstein gradient flows of functionals involving the entropic optimal transport functional $E$.

\subsection{Displacement smoothness and gradient flows}

Let $\bmu^0,\bmu^1,\gamma, \bmu^t$ and $\bphi^t$ be as in the beginning of paragraph \ref{sec-main-result}. A consequence of Thm.~\ref{thm:main} is that the functional $E$ is as nice as one could hope for in the Wasserstein space.

\begin{theorem}\label{thm:smoothness}
If $c\in \Cc^{2k-1}(\Xx)$ for some $k\geq 1$, then the function $t\mapsto E(\bmu^t)$ is of class $\Cc^k$. Moreover, if $c\in \Cc^{2}(\Xx)$ then its derivative is $C\cost(\gamma)$-Lipschitz, for some $C>0$ that only depends on $N$ and $\Vert c\Vert_{\Cc^2}$. In particular both $E$ and $-E$ are $(-C)$-displacement convex.
\end{theorem}
\begin{proof}
It follows from the dual formulation \eqref{eq:dual-MEOT} that $\bmu \mapsto E(\bmu)$ is a convex function of $\otimes_i \mu_i$ (eventhough $E$ is not convex) by optimality of $\bphi^t$  in this dual formulation $E$, setting $V_t(x):=e^{-c(x)+ \sum_{i=1}^N \phi_i^t(x_i)}$ it holds that for every $t$ and $s$ in $[0,1]$, one has
\begin{equation}\label{eq:convexity-differentiability}
\sum_i \int_{\Xx_i}\phi_i^t (\mu_i^{s}-\mu_i^t) -\int_{\Xx} V_t (\otimes_i \mu_i^s-\otimes_i \mu_i^t)\leq  E(\bmu^s)-E(\bmu^t) \leq \sum_i \int_{\Xx_i}\phi_i^s (\mu_i^{s}-\mu_i^t) -\int_{\Xx} V_s (\otimes_i \mu_i^s-\otimes_i \mu_i^t)
\end{equation}
Using the notation $x_i^t=(1-t)x_i+ty_i$ as before, remark that 
\begin{align*}
\int_{\Xx_i} \phi_i^t\d (\mu_i^{s}-\mu_i^t) &= \int_{\Xx_i^2} \big(\phi_i^t(x_i^{s})-\phi_i^t(x_i^t)\big)\d\gamma_i(x_i,y_i)\\
&= (s-t) \int_{\Xx_i^2} (y_i-x_i)^\top \nabla \phi_i^t(x_i^t)\d\gamma_i(x_i,y_i)+o(\vert s-t\vert).
\end{align*}
We now claim that the second term in the left hand side of \eqref{eq:convexity-differentiability} is $o(\vert s-t\vert)$. To prove this, we shall for notational simplicity restrict ourselves to the case $N=2$ (the general case is similar but  more tedious),
\[\begin{split}
    \int_{\Xx} V_t (\mu_1^s\otimes \mu_2^s- \mu_1^t\otimes \mu_2^t)= \int_{\Xx} V_t (\mu^t_{1} \otimes (\mu^s_{2}-\mu^t_2)+ (\mu^s_1-\mu^t_1)\otimes \mu_2^t)+ \int_{\Xx} V_t (\mu_1^s-\mu_1^t) \otimes (\mu_2^s-\mu_2^t)
\end{split}\]
the first term in the right-hand side is $0$ because the integral of $V_t(., x_2)$ (respectively $V_t(x_1, .)$ with respect to $\mu_1^t$ (respectively $\mu_2^t$) is constant equal to $1$ and $\mu_2^s$ and $\mu_2^t$ (respectively $\mu_1^s$ and $\mu_1^t$) have the same total mass.
We are therefore left to show that the second term is $o(\vert t-s\vert)$. Defining for $x_1\in \Xx_1$, $\xi_{s,t}(x_1):=\int_{\Xx_2}V_t(x_1,.) (\mu_2^s-\mu_2^t)$ and observing that since the $1$-Wasserstein distance between $\mu_2^s$ and $\mu_2^t$ is bounded by $M\vert t-s\vert$ where $M$  is the diameter of $\Xx_2$,  it follows from the Kantorovich-Rubinstein inequality that 
\[\vert\int_{\Xx} V_t (\mu_1^s-\mu_1^t) \otimes (\mu_2^s-\mu_2^t) \vert \leq M \vert t-s\vert\ \Vert \nabla \xi_{s,t} \Vert_{\Cc^0(\Xx_1)}.\]
Writing $\nabla \xi_{s,t}(x_1)$ as
\[\nabla \xi_{s,t} (x_1)=\nabla \phi_1^t(x_1) \xi_{t,s}(x_1)-\int_{\Xx_2} V_t(x_1,.) \nabla_{x_1} c(x_1,.) (\mu_2^s-\mu_2^t)\]
we deduce from the uniform continuity of $\nabla_{x_1} c$ and $V_t$ and the weak $*$ continuity of $s\mapsto \mu_2^s$ that $\nabla \xi_{s,t}$ converges uniformly to $0$ as $s\to t$, hence that 
\[ \int_{\Xx} V_t (\mu_1^s-\mu_1^t) \otimes (\mu_2^s-\mu_2^t)=o(\vert s-t\vert).\]
Thus dividing Eq.~\eqref{eq:convexity-differentiability} by $\vert s-t\vert$ and using that $\bphi^{s}\to \bphi^t$ in $\tilde \Cc^1$ as $s$ tends to $t$, we get:
\begin{align}\label{derivet}
     \frac{\d}{\d t} E(\bmu^t) = \sum_{i=1}^N \int_{\Xx_i \times \Xx_i} (y_i-x_i)^\top \nabla \phi^t_i(x^t_i)\d\gamma_i(x_i,y_i).
\end{align}
From Thm.~\ref{thm:main}, we know that $t\mapsto \bphi^t$ is in $\Cc^{k-1}([0,1],\tilde \Cc^{k})$ (note that the case $k=1$ is instead a consequence of Prop.~\ref{prop:preliminaries}) and hence $t\mapsto((x_i,y_i)\mapsto \nabla \bphi^t_i(x_i^t))$ is in $\Cc^{k-1}([0,1],\Cc^{0}(\Xx_i\times \Xx_i))$.

It follows that $h:t\mapsto E(\bmu^t)\in \Cc^k([0,1])$.  Notice how this argument uses the two notions of regularity of the Schr\"odinger map (indexed by $p$ and $k$ in Thm.~\ref{thm:main}).

For the Lipschitz regularity of $h'$, fixing $s,t\in [0,1]$, one has 
\begin{align*}
    \vert h'(t)-h'(s)\vert 
    &\leq \left\vert \sum_{i=1}^N \int_{\Xx_i \times \Xx_i} (y_i-x_i)^\top \big(\nabla\phi^t_i(x^t_i) - \nabla\phi^s_i(x^s_i)\big) \d\gamma_i(x_i,y_i)\right\vert \\
    &\leq \sum_{i=1}^N \int_{\Xx_i \times \Xx_i}  \Vert y_i - x_i \Vert \Vert \nabla\phi^t_i(x^t_i) - \nabla\phi^t_i(x^s_i) \Vert  \d\gamma_i(x_i,y_i)  \\
    &\qquad + \sum_{i=1}^N \int_{\Xx_i \times \Xx_i}  \Vert y_i - x_i \Vert \Vert  \nabla\phi^t_i(x^s_i) - \nabla\phi^s_i(x^s_i) \Vert \d\gamma_i(x_i,y_i).  
\end{align*}
Now, if $c\in \Cc^2(\Xx)$, using the Lipschitz regularity of $x_i\mapsto \nabla \phi_i^t(x_i)$ and of $t\mapsto \nabla \bphi^t$, from Thm.~\ref{thm:main}, it follows that
\[
\Vert\nabla\phi^t_i(x^t_i) - \nabla\phi^t_i(x^s_i) \Vert \leq C | t-s | \Vert y_i - x_i \Vert \text{  and  } \Vert \nabla\phi^t_i(x^s_i) - \nabla\phi^s_i(x^s_i) \Vert \leq C |t-s| \sqrt{\cost(\gamma)}, 
\]
for some $C$ that only depends on $N$ and $\Vert c\Vert_{\Cc^2}$. Then, we obtain
\begin{align*}
    \vert h'(t)-h'(s)\vert & \leq C |t-s| \sum_{i=1}^N \int_{\Xx_i \times \Xx_i}  \Vert y_i - x_i \Vert^2  \d\gamma_i(x_i,y_i)  \\
    &\qquad +  C|t-s|\sqrt{\cost(\gamma)} \sum_{i=1}^N \int_{\Xx_i \times \Xx_i}  \Vert y_i - x_i \Vert  \d\gamma_i(x_i,y_i)  \\
    & \leq C |t-s| \cost(\gamma).
\end{align*}
 
In particular, this implies that $t\in [0,1]\mapsto h(t)+ \frac{C \cost(\gamma)}{2} t^2$ is convex hence
\[ E(\bmu^t) \leq (1-t)  E(\bmu^0)+ t  E(\bmu^1) +\frac{C \cost(\gamma)t(1-t)}{2}\]
and displacement semi-convexity follows by choosing $\gamma_i$ to be an optimal transport plan between $\mu_i^0$ and $\mu_i^1$ for each $i\in [N]$ (see~\cite{ambrosio2005gradient} for a definition). Displacement semi-convexity of $-E$ is obtained in the same way, observing that $t\in [0,1]\mapsto -h(t)+ \frac{C\cost(\gamma)}{2} t^2$ is convex.
\end{proof}

\begin{proposition}
If $c\in \Cc^1(\Xx)$, we have that $S(\bmu)$ is the gradient of $\bmu\mapsto E(\bmu)$ and $x\mapsto \nabla_x S(\bmu)(x)$ is its Wasserstein gradient, in the sense of~\cite{ambrosio2005gradient} i.e. for $\bmu^0$ and $\bmu^1$ in  $\prod_{i=1}^N\Pp(\Xx_i)$ and for any $\gamma_i$ optimal plan between $\mu^0_i$, and $\mu_i^1$, one has
\[E(\bmu^1)-E(\bmu^0)= \sum_{i=1}^N \int_{\Xx_i \times \Xx_i} (y_i-x_i)^\top \nabla\phi_i(x_i) \d\gamma_i(x_i,y_i) +o(\bW(\bmu^0, \bmu^1)) \]
where  $\bphi:=S(\bmu^0)$. If $c\in \Cc^2(\Xx)$ then the error  $o(\bW(\bmu^0, \bmu^1))$ is in fact $O(\bW(\bmu^0, \bmu^1)^2)$.
\end{proposition}
\begin{proof}
For the case $c\in \Cc^1(\Xx)$, this follows by integrating~\eqref{derivet} in time and Proposition~\ref{prop:preliminaries}-(ii) which guarantees that $S$ is weakly continuous as a function in $\tilde \Cc^1$. For the case $c\in \Cc^2(\Xx)$, this follows from~\eqref{derivet} and  Thm.~\ref{thm:main}.
\end{proof}
 Thm. \ref{thm:smoothness} and the above identification of the Wasserstein gradient of $E$ enable us to deduce from~\cite[Thm.~11.2.1]{ambrosio2005gradient} that $E$ admits a unique Wasserstein gradient flow, which shows well-posedness of the Cauchy problem for the system of PDEs
 \[
 \left\{
 \begin{aligned}
 \partial_t \mu_i &= \nabla \cdot (\mu_i \nabla S_i(\bmu)),\quad i\in \{1, \ldots, N\}\\
 \bmu\vert_{t=0}&= \bmu^0
 \end{aligned}
 \right.
 \]
This system, as all the PDEs below, is understood in the sense of distributions with no-flux boundary conditions, i.e.~for $i\in\{1,\dots,N\}$, for every $\psi\in \Cc^\infty_c([0,+\infty)\times \RR^d)$ it holds
 $$
 \int_0^\infty \int_{\Xx_i} \big(\partial_t \psi(t,x_i) + \nabla S_i(\bmu)(x_i)^\top \nabla_{x_i} \psi(t,x_i)\big)\d\mu^t_i(x_i)\d t=-\int_{\Xx_i} \psi(0,x_i)\d\mu_i^0(x_i).
 $$
We also have that the fact that the gradient flow map $\bmu^0 \mapsto \bmu^t$ satisfies
\[\bW(\bmu^t, \bnu^t)^2 \leq e^{C t} \bW(\bmu^0, \bnu^0)^2.\]
Of course, adding to $E$ a separable term of the form $\sum_{i=1}^N E_i(\mu_i)$ where each $E_i$ is displacement semi-convex, we can deduce well-posedness for more general systems  like
 
 \[\partial_t \mu_i - \alpha_i \Delta \mu_i- \nabla \cdot (\mu_i \nabla S_i(\bmu))=0, i=1, \ldots, N, \quad \bmu\vert_{t=0}= \bmu^0\]
or 
 \[\partial_t \mu_i - \alpha_i \Delta \mu_i^{m_i}- \nabla \cdot (\mu_i \nabla S_i(\bmu))=0, i=1, \ldots, N, \quad \bmu\vert_{t=0}= \bmu^0\]
 with $m_i\geq 1$ and $\alpha_i \geq 0$. For the sake of concreteness, we are going to detail three such examples with interesting additional structure in the next paragraphs.

\subsection{Wasserstein gradient flow of the Sinkhorn divergence}

We consider the Sinkhorn divergence functional~\cite{feydy2019interpolating}, the gradient flow of which has been previously considered as a numerical method for density fitting. As a consequence of our analysis and of~\cite[Thm. 11.2.1]{ambrosio2005gradient} we have the following result.

\begin{proposition}
Let $\Xx\subset \RR^d$ be a compact convex set, $c\in \Cc^2(\Xx\times \Xx)$ and let $\mu^0,\nu\in \Pp(\Xx)$. There exists a unique Wasserstein gradient flow starting from $\mu^0$ of the Sinkhorn divergence (from $\nu$) functional 
$$
\mu \mapsto E(\mu,\nu) -\frac12 E(\mu,\mu) -\frac12 E(\nu,\nu).
$$
\end{proposition}
Here, a Wasserstein gradient flow is a curve $(\mu^t)_{t\geq 0}\in \Pp(\Xx)$ that is absolutely continuous for the $W_2$ metric and that satisfies
\begin{align}
\partial_t \mu^t = \nabla \cdot (v^t\mu^t), && v^t = \nabla S_1(\mu^t,\nu) - \frac{1}{2}(\nabla S_1(\mu^t,\mu^t)+\nabla S_2(\mu^t,\mu^t) ),&& \mu\vert_{t=0}=\mu^0
\end{align}
where we recall that $S$ is the Schr\"odinger map.

An interesting open question is whether this dynamics can be provably shown to converge to the unique minimizer $\mu^*$, which is $\mu^*=\nu$ for suitable choices of costs, e.g.~for $c(x,y)= \Vert y-x\Vert^2$, as proved in~\cite{feydy2019interpolating}. 

\subsection{Convergence to equilibrium for the Schr\"odinger bridge energy}

Let us continue with another simple example that shows that our theory is also natural to deal with the Lebesgue measure as a reference in the definition of $E$ Eq.~\eqref{eq:EOT}, which is the original definition of the Schr\"odinger bridge problem. This alternative definition is equivalent (see e.g.~\cite{marino2020optimal}) to considering $E+H$
where $H(\mu)\coloneqq \int \log(\mu)\d\mu$ if $\mu$ is absolutely continuous and $+\infty$ otherwise is minus the differential entropy.

\begin{proposition}\label{prop: GF + equilibrium 1 species}
Let $\Xx\subset \RR^d$ be a compact convex set, $c\in \Cc^2(\Xx\times \Xx)$ and let $\mu^0,\nu\in \Pp(\Xx)$. There exists a unique Wasserstein gradient flow of 
$$
\mu \mapsto E(\mu,\nu) + H(\mu)
$$
starting from $\mu^0$. Moreover, if $H(\mu^0)<\infty$ then this gradient flow converges at an exponential rate to the unique global minimizer $\mu^*$. Specifically, there exists $\kappa>0$ independent of $\mu^0,\nu$ such that
$$
F(\mu^t)-F(\mu^*)\leq e^{-\kappa t}(F(\mu^0)-F(\mu^*)).
$$
In addition, there exists a constant $C>0$, independent of $\mu^0$, $\nu$ such that 
 \[
 W_2(\mu^t, \mu^*)^2 \leq Ce^{-\kappa t}(F(\mu^0) - F(\mu^*)).
 \]
\end{proposition}
For this functional, the Wasserstein gradient flow $(\mu^t)_{t\geq 0}\in \Pp(\Xx)$ solves
\begin{align}
\partial_t \mu^t = \nabla \cdot (v^t\mu^t) +\Delta \mu^t, && v^t = \nabla S_1(\mu^t,\nu).
\end{align}

\begin{proof}
We have semi-convexity along Wasserstein geodesics, by Thm.~\ref{thm:smoothness} for the first component and by a standard result due to Mc Cann~\cite{mccann1997convexity} (see~\cite[Thm.~7.28]{santambrogio2015optimal}) for the  $H$ component. Thus the general well-posedness results from~\cite[Thm.~11.2.1]{ambrosio2005gradient} applies. For the exponential convergence -- in function value and in distance -- we apply the result from~\cite[Thm. 3.2]{chizat2022mean}, see also~\cite{nitanda2022convex} where the same argument was discovered independently (although stated on $\RR^d$, the argument goes through on a compact domain). 

The main assumptions to check are that (i) $\mu\mapsto E(\mu,\nu)$ is convex, which is clear from the dual formulation Eq.~\eqref{eq:dual-MEOT} which expresses this functional as a supremum of affine forms, (ii) that a global minimizer $\mu^*$ exists, which is not difficult here since $\Pp(\Xx)$ is weakly compact, $H$ is weakly lower-semicontinuous and $E$ is weakly continuous and finally we need to check that the probability measure $\hat \mu_t \propto e^{-S_1(\mu^t,\nu)} \in \Pp(\Xx)$ satisfies a log-Sobolev inequality, uniformly in $t$ (Assumption 3 in~\cite{chizat2022mean}). 

Since $\Xx$ is bounded, the normalized Lebesgue measure satisfies a log-Sobolev inequality~\cite[Thm. 7.3]{ledoux2001logarithmic}. By the Holley-Stroock perturbation criterion~\cite{holley1986logarithmic} (see~\cite[Lem.~1.2]{ledoux2001logarithmic}), $\hat \mu_t$ satisfies it as well; this criterion
applies here because $\sup_x S_1(\mu,\nu)(x) - \inf_x S_1(\mu,\nu)(x)$ is bounded, uniformly in $\mu,\nu \in \Pp(\Xx)$ by Prop.~\ref{prop:preliminaries}. The convergence in Wasserstein distance is stated in~\cite[Cor.~3.3]{chizat2022mean} and follows from
the fact that log-Sobolev inequalities imply Talagrand inequalities~\cite[Thm.~1]{otto2000generalization}.

\end{proof}

\subsection{Convergence to equilibrium in the multi-species case}

 We now consider Wasserstein gradient flow of\footnote{Note that $F(\bmu)$ can also be written as the value of an entropic optimal transport problem but with  the Lebesgue measure as reference measure i.e. $F(\bmu)= \min_{\gamma\in \Pi(\bmu)} \int_\Xx c(x)\d\gamma(x) + H(\gamma)$.}
 \[F(\bmu):=E(\bmu)+ \sum_{i=1}^N H(\mu_i), \quad \bmu=(\mu_1, \ldots, \mu_N)\in  \prod_{i=1}^N\Pp(\Xx_i)\]
 that is 
 \begin{equation}\label{eq:multi-spec-H}
 \partial_t \mu_i -\Delta \mu_i-\nabla \cdot (\mu_i \nabla S_i(\bmu))=0, \quad i=1, \ldots, N, \; \bmu\vert_{t=0}=\bmu^0.
 \end{equation}
 Up to adding a constant to $c$ (which does not affect the dynamics \eqref{eq:multi-spec-H}) we may assume that
 \begin{equation}\label{eq:normalizc}
 \int_{\Xx} e^{-c(x)}\d x=1.   
 \end{equation}
 With this normalization, we have
 \[\inf_{\bmu \in  \prod_{i=1}^N\Pp(\Xx_i) } F(\bmu)=\inf_{\gamma\in \Pp(\Xx)} H(\gamma\vert e^{-c})=0\]
 so that $F$ admits $\bmu^*$, the marginals of $\gamma^*:=e^{-c}$, as unique minimizer    
 \begin{equation}\label{defdemustar}
 \mu_i^*(x_i)=\int_{\Xx_{-i}} e^{-c(x_i, x_{-i})} \d x_{-i}
 \end{equation}
 for every $i$ and $x_i \in \Xx_i$ and $F(\bmu^*)=0$. In the next proposition, we extend Prop.~\ref{prop: GF + equilibrium 1 species} to the multi-species case. In this case, the functional $\bmu \mapsto E(\bmu)$ is not convex anymore but we can overcome this difficulty taking advantage of the form of $F$. 
 
 \begin{proposition}\label{prop:multi-species}
 Let $\bmu^0 \in \prod_{i=1}^N\Pp(\Xx_i)$. Then there exists a unique Wasserstein gradient flow of $F$ starting from $\bmu^0$, that we call $\bmu^t$. Assume that $H(\mu^0_i)<+\infty$ for every $i$, then $\bmu^t$ converges at an exponential rate to the equilibrium $\bmu^*$, defined in Eq.~\eqref{defdemustar}, i.e. there exists $\kappa >0$ independent of $\bmu^0$ such that 
 \[
 F(\bmu^t) - F(\bmu^*) \leq e^{-\kappa t}(F(\bmu^0) - F(\bmu^*)).
 \]
 In addition, there exists a constant $C>0$, independent of $\bmu^0$, such that 
 \[
 \bW(\bmu^t, \bmu^*)^2 \leq Ce^{-\kappa t}(F(\bmu^0) - F(\bmu^*)).
 \]
 \end{proposition}
 
 \begin{proof}
 The well-posedness of the Wasserstein gradient flow is proved as previously using the geodesic semi-convexity of $F$ that follows from Thm.~\ref{thm:smoothness}. For the convergence, first note the identities
 \[E(\bmu)=\sum_{i=1}^N \int_{\Xx_i} S_i(\bmu) \d \mu_i, \quad F(\bmu)= \sum_{i=1}^N H(\mu_i \vert  e^{-S_i(\bmu)})\]
  which hold for any $\bmu \in \prod_{i=1}^N\Pp(\Xx_i)$ and easily follow from \eqref{eq:dual-MEOT} and \eqref{eq:schrodinger-system}. Let us then remark that, denoting by $\gamma(\bmu)$, the optimal entropic plan
  \[\d\gamma(\bmu)(x):=e^{-c(x)+ \sum_{i=1}^N S_i(\bmu)(x_i)} \d \mu_1(x_1) \cdots \d \mu_N( x_N)\]
and recalling that $\gamma^*=e^{-c}$  we can conveniently rewrite $F$ as a relative entropy with respect to the fixed probability measure $\gamma^*$ on $\Xx$:
 \[F(\bmu)=H(\gamma(\bmu)\vert \gamma^*).\]
Since $H(\mu^0_i)<+\infty$ for every $i$, and denoting by $\bmu^t$ the Wasserstein gradient flow of $F$ starting from $\bmu^0$, \rev{we have using the chain rule, \eqref{eq:multi-spec-H} and an integration by parts:
 \[\begin{split}
     \frac{\d}{\d t} F(\bmu^t)&=\sum_{i=1}^N \int_{\Xx_i}(S_i(\bmu^t)+\log(\mu_i^t)) \partial_t \mu_i^t   =-\sum_{i=1}^N \int_{\Xx_i} \Vert \nabla \log \mu_i^t + \nabla S_i(\bmu^t) \Vert^2  \d \mu_i^t\\
     &= -\sum_{i=1}^N I_i(\mu_i^t \vert e^{-S_i(\bmu^t)}) \end{split}\]
 }
 where $I_i(\rho \vert e^{-V})$, for $\rho \in \Pp(\Xx_i)$, stands for the relative Fisher information
 \[I_i(\rho \vert e^{-V}):=\int_{\Xx_i} \left\Vert \nabla \log\Big(\frac{\rho} {e^{-V}}\Big)\right\Vert^2 \d \rho.\]
Defining $\gamma^t:=\gamma(\bmu^t)$, we have $F(\bmu^t)=H(\gamma^t \vert \gamma^*)$ and 
\[\begin{split}
I(\gamma^t \vert \gamma^*)&:= \int_{\Xx} \left\Vert \nabla_x \log  \Big(\frac{\gamma^t(x)} {e^{-c(x)}} \Big)\right\Vert^2 \d \gamma^t(x)
=  \int_{\Xx} \sum_{i=1}^N  \left\Vert \nabla_{x_i} (\log(\mu^t_i(x_i)) + S_i(\bmu^t)(x_i))\right \Vert^2 \d \gamma^t(x)\\
&=\sum_{i=1}^N I_i(\mu_i^t \vert e^{-S_i(\bmu^t)})
\end{split} \]
where we used the fact that $\gamma^t\in \Pi(\bmu^t)$ in the last line. But since $\Xx$ is convex, it follows from the Holley-Stroock perturbation criterion~\cite{holley1986logarithmic} (see~\cite[Lem.~1.2]{ledoux2001logarithmic}), that $\gamma^*$ satisfies a log-Sobolev inequality, hence
 \[I(\gamma^t \vert \gamma^*) \geq  \kappa H(\gamma^t \vert \gamma^*)\]
 with $\kappa>0$ depending only on $\Xx$ and $c$.  We thus have
 \[\frac{\d}{\d t} F(\bmu^t)= \frac{\d}{\d t} H(\gamma^t\vert \gamma^*)=-I(\gamma^t \vert \gamma^*)   \leq -\kappa H(\gamma^t \vert \gamma^*)=-\kappa F(\bmu^t)\]
 hence
 \[F(\bmu^t) \leq e^{-\kappa t} F(\bmu^0).\]
Thanks to Talagrand's inequality, which follows from the log-Sobolev inequality~\cite[Thm.~1]{otto2000generalization}, we  get an exponential decay of $W_2(\gamma_t, \gamma^*)$ hence also an exponential decay in Wasserstein distance between the marginals of $\gamma^t$ and $\gamma^*$ i.e. of $\bW(\bmu^t, \bmu^*)$.
  \end{proof}

\bibliographystyle{plain}

\bibliography{bibli}

\end{document}